\documentclass[a4paper]{amsart}

\usepackage{a4,exscale,amsfonts,amssymb,amsmath,amscd}
\usepackage{nicefrac,tikz,tikz-3dplot,graphicx,color}
\usetikzlibrary{arrows,calc,positioning}
\usepackage[font=scriptsize]{caption}
\usepackage[english]{babel}

\allowdisplaybreaks

\usepackage[mathscr]{eucal}
\usepackage{color,amsmath}

\newcommand{\Abb}[5]{\begin{array}{ccccc}#1&:&#2&\longrightarrow&#3\\{}&{}&#4&\longmapsto&#5\end{array}}

\newcommand{\ol}{\overline}
\newcommand{\ul}{\underline}

\newcommand{\nz}{\mathbb{N}}

\newcommand{\rz}{\mathbb{R}}

\newcommand{\rN}{\rz^N}
\newcommand{\rNtN}{\rz^{N\times N}}

\DeclareMathOperator{\p}{\partial}

\newcommand{\na}{\nabla}

\DeclareMathOperator{\ed}{d}

\DeclareMathOperator{\cd}{\delta}

\newcommand{\T}{\mathcal{T}}
\renewcommand{\S}{\mathcal{S}}
\renewcommand{\P}{\mathcal{P}}

\newcommand{\Q}{\mathcal{Q}}
\newcommand{\E}{\mathcal{E}}

\newcommand{\eps}{\varepsilon}

\newcommand{\om}{\Omega}

\newcommand{\equi}{\Leftrightarrow}
\def\impl{\Rightarrow}

\newcommand{\qequi}{\quad\equi\quad}

\newcommand{\qimpl}{\quad\impl\quad}

\DeclareMathOperator{\supp}{supp}

\DeclareMathOperator{\dist}{dist}

\newcommand{\dvec}[3]{\begin{bmatrix}#1\\#2\\#3\end{bmatrix}}

\DeclareMathOperator{\id}{id}

\def\set#1#2{\{#1\,:\,#2\}}
\def\bset#1#2{\big\{#1\,:\,#2\big\}}

\DeclareMathOperator{\Lebesgue}{\mathsf{L}}
\newcommand{\Lgen}[2]{\Lebesgue^{#1}_{#2}}

\def\Lt{\Lgen{2}{}}

\def\Ltom{\Lt(\om)}

\def\Ltqeps{\Lgen{2,q}{\eps}}
\def\Ltqepsom{\Lgen{2,q}{\eps}(\om)}

\def\LtqepsXi{\Lgen{2,q}{\eps}(\Xi)}

\def\L6{\Lgen{6}{}}

\newcommand{\qLt}[1]{\Lgen{2,#1}{}}
\newcommand{\Ltq}{\qLt{q}}
\newcommand{\LtNmq}{\qLt{N-q}}
\newcommand{\Ltqpo}{\qLt{q+1}}
\newcommand{\Ltqmo}{\qLt{q-1}}
\newcommand{\Ltqom}{\Ltq(\om)}

\newcommand{\Ltqpoom}{\Ltqpo(\om)}
\newcommand{\Ltqmoom}{\Ltqmo(\om)}

\DeclareMathOperator{\DSobolev}{\mathsf{D}}
\newcommand{\Dgen}[3]{\overset{#3}{\DSobolev}{}^{#1}_{#2}}

\newcommand{\Dgenc}[2]{\mathring\DSobolev{}^{#1}_{#2}}
\newcommand{\qD}[1]{\Dgen{#1}{}{}}
\newcommand{\qDz}[1]{\Dgen{#1}{0}{}}
\newcommand{\qDc}[1]{\Dgenc{#1}{}}
\newcommand{\qDcgn}[1]{\Dgenc{#1}{\gamma_{\nu}}}
\newcommand{\qDcgt}[1]{\Dgenc{#1}{\gamma_{\tau}}}
\newcommand{\qDcn}[1]{\Dgenc{#1}{\Gamma_{\nu}}}
\newcommand{\qDct}[1]{\Dgenc{#1}{\Gamma_{\tau}}}

\newcommand{\qDcz}[1]{\Dgenc{#1}{0}}

\newcommand{\qDczt}[1]{\Dgenc{#1}{\Gamma_{\tau},0}}

\newcommand{\qDczgt}[1]{\Dgenc{#1}{\gamma_{\tau},0}}
\newcommand{\qDczgn}[1]{\Dgenc{#1}{\gamma_{\nu},0}}

\newcommand{\Dq}{\qD{q}}
\newcommand{\DNmqmo}{\qD{N-q-1}}

\newcommand{\Dqmo}{\qD{q-1}}

\newcommand{\Dqk}{\qD{k,q}}
\newcommand{\Dqomt}{\qD{1,q-2}}
\newcommand{\Dqz}{\qDz{q}}
\newcommand{\Dqpoz}{\qDz{q+1}}
\newcommand{\Dqmoz}{\qDz{q-1}}
\newcommand{\Dqc}{\qDc{q}}

\newcommand{\Dqmoc}{\qDc{q-1}}
\newcommand{\Dqcz}{\qDcz{q}}

\newcommand{\Dqct}{\qDct{q}}
\newcommand{\Dqcn}{\qDcn{q}}

\newcommand{\Dqmoct}{\qDct{q-1}}

\newcommand{\Dqmocn}{\qDcn{q-1}}

\newcommand{\Dqcgt}{\qDcgt{q}}
\newcommand{\Dqcgn}{\qDcgn{q}}
\newcommand{\Dqmocgt}{\qDcgt{q-1}}
\newcommand{\Dqmocgn}{\qDcgn{q-1}}

\newcommand{\Dqczt}{\qDczt{q}}
\newcommand{\Dqmoczt}{\qDczt{q-1}}

\newcommand{\Dqczgt}{\qDczgt{q}}
\newcommand{\Dqczgn}{\qDczgn{q}}

\newcommand{\Dqom}{\Dq(\om)}

\newcommand{\Dqmoom}{\Dqmo(\om)}
\newcommand{\Dqzom}{\Dqz(\om)}
\newcommand{\Dqpozom}{\Dqpoz(\om)}

\newcommand{\Dqcom}{\Dqc(\om)}

\newcommand{\Dqmocom}{\Dqmoc(\om)}
\newcommand{\Dqczom}{\Dqcz(\om)}

\DeclareMathOperator{\cDSobolev}{\boldsymbol\DSobolev}

\newcommand{\cDgenc}[2]{\mathring\cDSobolev{}^{#1}_{#2}}

\newcommand{\cqDc}[1]{\cDgenc{#1}{}}

\newcommand{\cqDct}[1]{\cDgenc{#1}{\Gamma_{\tau}}}
\newcommand{\cqDcn}[1]{\cDgenc{#1}{\Gamma_{\nu}}}

\newcommand{\cqDcgn}[1]{\cDgenc{#1}{\gamma_{\nu}}}

\newcommand{\cqDczt}[1]{\cDgenc{#1}{\Gamma_{\tau},0}}
\newcommand{\cqDczn}[1]{\cDgenc{#1}{\Gamma_{\nu},0}}

\newcommand{\cqDczgn}[1]{\cDgenc{#1}{\gamma_{\nu},0}}

\newcommand{\cDqc}{\cqDc{q}}

\newcommand{\cDqmocgn}{\cqDcgn{q-1}}

\newcommand{\cDqct}{\cqDct{q}}
\newcommand{\cDqcn}{\cqDcn{q}}
\newcommand{\cDqmocn}{\cqDcn{q-1}}
\newcommand{\cDqczt}{\cqDczt{q}}
\newcommand{\cDqczn}{\cqDczn{q}}

\newcommand{\cDqcgn}{\cqDcgn{q}}

\newcommand{\cDqczgn}{\cqDczgn{q}}

\newcommand{\cDqpoczgn}{\cqDczgn{q+1}}

\newcommand{\cDqmoczn}{\cqDczn{q-1}}

\DeclareMathOperator{\DeSobolev}{\Delta}
\newcommand{\Degen}[3]{\overset{#3}{\DeSobolev}{}^{#1}_{#2}}
\newcommand{\Degenc}[2]{\mathring\DeSobolev{}^{#1}_{#2}}
\newcommand{\qDe}[1]{\Degen{#1}{}{}}
\newcommand{\qDec}[1]{\Degenc{#1}{}}
\newcommand{\qDez}[1]{\Degen{#1}{0}{}}
\newcommand{\qDecz}[1]{\Degenc{#1}{0}}

\newcommand{\qDecn}[1]{\Degenc{#1}{\Gamma_{\nu}}}

\newcommand{\qDecgn}[1]{\Degenc{#1}{\gamma_{\nu}}}

\newcommand{\qDeczn}[1]{\Degenc{#1}{\Gamma_\nu,0}}

\newcommand{\qDeczgn}[1]{\Degenc{#1}{\gamma_{\nu},0}}

\newcommand{\Deq}{\qDe{q}}

\newcommand{\Deqk}{\qDe{k,q}}
\newcommand{\Deqpo}{\qDe{q+1}}

\newcommand{\Deqc}{\qDec{q}}
\newcommand{\Deqpoc}{\qDec{q+1}}

\newcommand{\Deqz}{\qDez{q}}

\newcommand{\Deqmoz}{\qDez{q-1}}
\newcommand{\Deqcz}{\qDecz{q}}

\newcommand{\Deqom}{\Deq(\om)}
\newcommand{\Deqpoom}{\Deqpo(\om)}

\newcommand{\Deqcom}{\Deqc(\om)}
\newcommand{\Deqpocom}{\Deqpoc(\om)}

\newcommand{\Deqzom}{\Deqz(\om)}

\newcommand{\Deqmozom}{\Deqmoz(\om)}
\newcommand{\Deqczom}{\Deqcz(\om)}

\newcommand{\Deqcn}{\qDecn{q}}

\newcommand{\Deqcgn}{\qDecgn{q}}

\newcommand{\Deqpocgn}{\qDecgn{q+1}}

\newcommand{\Deqczn}{\qDeczn{q}}
\newcommand{\Deqpoczn}{\qDeczn{q+1}}

\newcommand{\Deqczgn}{\qDeczgn{q}}

\newcommand{\Deqpocn}{\qDecn{q+1}}

\DeclareMathOperator{\cDeSobolev}{\boldsymbol\DeSobolev}

\newcommand{\cDegenc}[2]{\mathring\cDeSobolev{}^{#1}_{#2}}

\newcommand{\cqDec}[1]{\cDegenc{#1}{}}

\newcommand{\cqDecn}[1]{\cDegenc{#1}{\Gamma_{\nu}}}

\newcommand{\cqDecgn}[1]{\cDegenc{#1}{\gamma_{\nu}}}

\newcommand{\cqDeczn}[1]{\cDegenc{#1}{\Gamma_\nu,0}}

\newcommand{\cqDeczgn}[1]{\cDegenc{#1}{\gamma_{\nu},0}}

\newcommand{\cDeqc}{\cqDec{q}}

\newcommand{\cDeqcn}{\cqDecn{q}}

\newcommand{\cDeqcgn}{\cqDecgn{q}}

\newcommand{\cDeqpocgn}{\cqDecgn{q+1}}

\newcommand{\cDeqczn}{\cqDeczn{q}}

\newcommand{\cDeqczgn}{\cqDeczgn{q}}

\DeclareMathOperator{\Sobolev}{\mathsf{H}}
\newcommand{\Hgen}[3]{\overset{#3}{\Sobolev}{}^{#1}_{#2}}
\newcommand{\Hgenc}[2]{\mathring\Sobolev{}^{#1}_{#2}}
\newcommand{\Hgentr}[2]{\mathring{\utilde{\Sobolev}}{}^{#1}_{#2}}

\def\Ho{\Hgen{1}{}{}}

\def\Hk{\Hgen{k}{}{}}
\def\Hoom{\Ho(\om)}

\def\Hkom{\Hk(\om)}

\def\Hoq{\Hgen{1,q}{}{}}
\def\Hoqmo{\Hgen{1,q-1}{}{}}
\def\Hoqpo{\Hgen{1,q+1}{}{}}
\def\Hoqmt{\Hgen{1,q-2}{}{}}

\def\Hoqom{\Hoq(\om)}
\def\Hoqmoom{\Hoqmo(\om)}
\def\Hoqpoom{\Hoqpo(\om)}

\def\Hkq{\Hgen{k,q}{}{}}
\def\Hkqpo{\Hgen{k,q+1}{}{}}
\def\Hkqmo{\Hgen{k,q-1}{}{}}

\def\Hkqom{\Hkq(\om)}

\def\Hoqmot{\Hgenc{1,q-1}{\gamma_\tau}}
\def\Hoqpon{\Hgenc{1,q+1}{\gamma_\nu}}
\def\Hoqcgn{\Hgenc{1,q}{\gamma_\nu}}

\def\Hoqcn{\Hgenc{1,q}{\Gamma_\nu}}
\def\Hoqct{\Hgenc{1,q}{\Gamma_\tau}}
\def\Hoqmocn{\Hgenc{1,q-1}{\Gamma_\nu}}

\def\Hoqmocgn{\Hgenc{1,q-1}{\gamma_\nu}}
\def\Hoqpocgn{\Hgenc{1,q+1}{\gamma_\nu}}
\def\Hoqmocgt{\Hgenc{1,q-1}{\gamma_\tau}}

\def\Hoqcttr{\Hgentr{1,q}{\Gamma_\tau}}

\def\Hoqmocgntr{\Hgentr{1,q-1}{\gamma_\nu}}

\DeclareMathOperator{\cSobolev}{\boldsymbol\Sobolev}

\newcommand{\cHgenc}[2]{\mathring\cSobolev{}^{#1}_{#2}}

\def\cHoqct{\cHgenc{1,q}{\Gamma_\tau}}

\DeclareMathOperator{\Cont}{\mathsf{C}}
\newcommand{\Cgen}[2]{\overset{#2}{\Cont}{}^{#1}}
\newcommand{\Cgenc}[1]{\mathring\Cont{}^{#1}}

\def\Czo{\Cgen{0,1}{}}
\def\Czoq{\Cgen{0,1,q}{}}

\def\Czoqct{\Cgenc{0,1,q}_{\Gamma_\tau}}

\def\Cic{\Cgenc{\infty}}

\def\Cicn{\Cic_{\Gamma_\nu}}

\def\Ciq{\Cgen{\infty,q}{}}
\def\Ciqc{\Cgenc{\infty,q}}
\def\CiNmqmoc{\Cgenc{\infty,N-q-1}}
\def\Citc{\Cgenc{\infty,2}}

\def\Ciqpoc{\Cgenc{\infty,q+1}}

\def\Ciqom{\Ciq(\om)}

\def\Ciqct{\Cgenc{\infty,q}_{\Gamma_\tau}}
\def\Ciqcn{\Cgenc{\infty,q}_{\Gamma_\nu}}
\def\Ciqpocn{\Cgenc{\infty,q+1}_{\Gamma_\nu}}
\def\Ciqmoct{\Cgenc{\infty,q-1}_{\Gamma_\tau}}

\DeclareMathOperator{\dirichlet}{\mathcal{H}}
\newcommand{\qharmdi}[2]{\dirichlet^{#1}_{#2}(\om)}

\newcommand{\harmdiq}{\qharmdi{q}{}}

\newcommand{\harmdiqeps}{\qharmdi{q}{\eps}}

\newcommand{\normdst}{\hspace{-0.4ex}}
\newcommand{\scp}[2]{\langle#1,#2\rangle}

\newcommand{\scpLtom}[2]{\scp{#1}{#2}_{\Ltom}}

\newcommand{\scpLtqom}[2]{\scp{#1}{#2}_{\Ltqom}}

\newcommand{\scpLtqpoom}[2]{\scp{#1}{#2}_{\Ltqpoom}}

\newcommand{\scpLtqmoom}[2]{\scp{#1}{#2}_{\Ltqmoom}}

\newcommand{\scpLtqXi}[2]{\scp{#1}{#2}_{\Ltq(\Xi)}}

\newcommand{\scpLtqepsom}[2]{\scp{#1}{#2}_{\Ltqepsom}}
\newcommand{\scpLtqepsXi}[2]{\scp{#1}{#2}_{\LtqepsXi}}

\newcommand{\scpLtqmoXi}[2]{\scp{#1}{#2}_{\Ltqmo(\Xi)}}

\newcommand{\norm}[1]{|#1|}

\newcommand{\norms}[1]{|\normdst|#1|\normdst|}

\newcommand{\normLtqom}[1]{\norm{#1}_{\Ltqom}}

\newcommand{\normLtqpoom}[1]{\norm{#1}_{\Ltqpoom}}

\newcommand{\normLtqmoom}[1]{\norm{#1}_{\Ltqmoom}}

\newcommand{\normLtqXi}[1]{\norm{#1}_{\Ltq(\Xi)}}
\newcommand{\normLtqepsXi}[1]{\norm{#1}_{\Ltqeps(\Xi)}}
\newcommand{\normLtqmoXi}[1]{\norm{#1}_{\Ltqmo(\Xi)}}

\newcommand{\normHoqpoxi}[1]{\norm{#1}_{\Hoqpo(\Xi)}}
\newcommand{\normHoqmoxi}[1]{\norm{#1}_{\Hoqmo(\Xi)}}

\newcommand{\normDqkom}[1]{\norm{#1}_{\Dqk(\om)}}
\newcommand{\normDqkrn}[1]{\norm{#1}_{\Dqk(\rz^N)}}

\newtheorem{lem}{Lemma}[section]
\newtheorem{defi}[lem]{Definition}
\newtheorem{theo}[lem]{Theorem}

\newtheorem{cor}[lem]{Corollary}

\newtheorem{rem}[lem]{Remark}

\def\undertilde#1{\mathord{\vtop{\ialign{##\crcr
$\hfil\displaystyle{#1}\hfil$\crcr\noalign{\kern1.5pt\nointerlineskip}
$\hfil\widetilde{}\hfil$\crcr\noalign{\kern1.5pt}}}}}
\def\wideundertilde#1#2{\mathord{\vtop{\ialign{##\crcr
$\hfil\displaystyle{#2}\hfil$\crcr\noalign{\kern2pt\nointerlineskip}
$\hfil\widetilde{\hspace*{#1mm}}\hfil$\crcr\noalign{\kern2pt}}}}}
\def\utilde#1{\undertilde{#1}}

\title[Weck's Selection Theorem]
{Weck's Selection Theorem:
The Maxwell Compactness Property
for Bounded Weak Lipschitz Domains
with Mixed Boundary Conditions in Arbitrary Dimensions}
\author{Sebastian Bauer}
\author{Dirk Pauly}
\author{Michael Schomburg}
\address{Fakult\"at f\"ur Mathematik,
Universit\"at Duisburg-Essen, Campus Essen, Germany}
\email[Sebastian Bauer]{sebastian.bauer.seuberlich@uni-due.de}
\email[Dirk Pauly]{dirk.pauly@uni-due.de}
\email[Michael Schomburg]{michael.schomburg@uni-due.de}
\date{\today; {\it Corresponding Author}: Dirk Pauly}
\thanks{We cordially thank Immanuel Anjam for providing the graphics in this paper.}

\begin{document}


\begin{abstract}
It is proved that the space of differential forms with weak exterior- and co-derivative, 
is compactly embedded into the space of square integrable  differential forms. 
Mixed boundary conditions on weak Lipschitz domains are considered. 
Furthermore, canonical applications such as Maxwell estimates, 
Helmholtz decompositions and a static solution theory are proved.
As a side product and crucial tool for our proofs
we show the existence of regular potentials and regular decompositions as well.
\end{abstract}


\maketitle
\begin{center}
In Memoriam of our Dear Friend and Mentor Karl-Josef (Charlie) Witsch (1948-2017)
\end{center}
\vspace*{2mm}
\tableofcontents


\section{Introduction}

The aim of this contribution is to prove a compact embedding, 
so called ``Weck's selection theorem''
or (generalized) Maxwell compactness property \cite{weckmaxcomp,weckmax,picardcomimb}, 
of differential $q$-forms with weak exterior and co-derivative into the space of square integrable $q$-forms
subject to mixed boundary conditions on bounded weak Lipschitz domains $\om\subset\rN$, i.e.,
\begin{align*}
\Dqct(\om)\cap\eps^{-1}\Deqcn(\om)\hookrightarrow\Ltqom
\end{align*}
is compact.
The main result is given by Theorem \ref{satzMKE}.
Here $N\geq2$ and $0\leq q\leq N$
are natural numbers, the dimension of the domain $\om$
and the rank of the differential forms, respectively.
This generalises the results from \cite{bauerpaulyschomburgmaxcompweaklip}, 
where bounded  weak Lipschitz domains in the classical setting of $\rz^3$ were considered.
In fact, the results from \cite{bauerpaulyschomburgmaxcompweaklip} can be recovered by setting 
$N=3$ and $q=1$ or $q=2$.  

Similar results for strong Lipschitz domains in three dimensions can be found in \cite{jochmanncompembmaxmixbc,FeGi97}.
For a historical overview of the mathematical treatment of Weck's selection theorem (Maxwell compactness property) 
see \cite{bauerpaulyschomburgmaxcompweaklip,leisbook,picardweckwitschxmas} 
and the literature cited therein.
In particular, let us mention the important papers
\cite{weckmaxcomp,webercompmax,picardcomimb,costabelremmaxlip,witschremmax,jochmanncompembmaxmixbc,picardweckwitschxmas}. 
We emphasise that in \cite{witschremmax} Witsch was able to go even beyond 
Lipschitz regularity ($p$-cusps). In \cite{weckelacomp} Weck applied Witsch's ideas 
to the theory of elasticity.

The central role of compact embeddings of this type can for example be seen in connection with Hilbert space complexes,
where the compact embeddings immediately provide closed ranges, solution theories by continuous inverses, 
Friedrichs/Poincar\'e-type estimates,
and access to Hodge-Helmholtz-type decompositions, Fredholm theory, div-curl-type lemmas,
and a-posteriori error estimation, see \cite{paulyapostfirstordergen,paulydivcurl,paulyzulehnerbiharmonic}. 
In exterior domains, where local versions of the compact embeddings hold,
one obtains radiation solutions (scattering theory) with the help of
Eidus' limiting absorption principle \cite{eiduslabp,eiduslamp,eiduslamptwo}, 
see \cite{paulydiss,paulytimeharm,paulystatic,paulydeco,paulyasym,paulypoly}.
We elaborate on some of these applications in our Section \ref{sectApplications}.

Finally we note that by the same arguments as in \cite{picardcomimb} our results extend to Riemannian manifolds. 

\section{Notations, Preliminaries and Outline of the Proof}

Let $ \om\subset\rN $ be a bounded weak Lipschitz domain. For a precise definition of weak Lipschitz domains, 
see Definitions \ref{defilipmani} and \ref{defilipsubmani}. 
In short, $ \om$ is an $N$-dimensional $\Czo$-submanifold of $\rN$ with boundary, i.e., a manifold with Lipschitz atlas. 
Let $ \Gamma := \partial \om $, which is itself an $(N-1)$-dimensional Lipschitz-manifold without boundary, 
consist of two relatively open subsets $\Gamma_\tau$ and $\Gamma_\nu$ 
such that $\overline\Gamma_\tau \cup \overline\Gamma_\nu = \Gamma$ and $\Gamma_{\tau}\cap\Gamma_{\nu}=\emptyset$. 
The separating set $ \overline\Gamma_\tau \cap \overline\Gamma_\nu $ (interface)
will be assumed to be a, not necessarily connected, $(N-2)$-dimensional Lipschitz-submanifold of $\Gamma$. 
We shall call $(\om, \Gamma_{\tau})$ a weak Lipschitz pair.

We will be working in the framework of alternating differential forms, see for example \cite{Janich2001VA}.
The vector space $\Ciqc(\om)$ is defined as the subset of $\Ciq(\om)$, the set of smooth alternating differential forms of rank $q$, having compact support in $\om$. Together with the inner product  
$$
\scpLtqom{E}{H} := \int_\om E \wedge \star H
$$
it is an inner product space\footnote{For simplicity we work in a real Hilbert space setting.}. We may then define $\Ltqom$ as the completion of $\Ciqc(\om)$ with respect to the corresponding norm.
$\Ltqom$ can be identified with those q-forms having $\Lt$-coefficients with respect to any coordinate system.
Using the weak version of Stokes' theorem
\begin{align}
\label{defweakderi}
\scpLtqpoom{\ed E}{H} = -\scpLtqom{E}{\cd H}, \qquad E\in\Ciqc(\om),\,H\in\Ciqpoc(\om),
\end{align}
weak versions of the exterior derivative and co-derivative can be defined. 
Here $\ed$ is the exterior derivative, $\cd=(-1)^{N(q-1)}\star\ed\star$ the co-derivative
and $\star$ the Hodge-star-operator on $\om$.
We thus introduce the Sobolev (Hilbert) spaces (equipped with their natural graph norms)
\begin{align*}
\Dq(\om):=\left\{E\in\Ltqom: \ed E \in\Ltqpoom\right\},\quad \Deq(\om):=\left\{E\in\Ltqom: \cd E \in\Ltqmo(\om)\right\}
\end{align*}
in the distributional sense. It holds
$$\star\Dqom=\Degen{N-q}{}{}(\om),\qquad \star\Deqom=\Dgen{N-q}{}{}(\om).$$

We further define the test forms
\begin{align*}
\Ciqct(\om) :=\bset{\varphi|_{\om}}{\varphi\in\Ciqc(\rN),~ \dist(\supp \varphi, \Gamma_\tau) > 0}
\end{align*}
and note that $ \Ciqc_\emptyset(\om)= \Ciq(\ol\om) $.
We now take care of boundary conditions. 
First we introduce strong boundary conditions as closures of test forms by
\begin{align}
\label{defstark}
\Dqct(\om) :=\overline{\Ciqct(\om)}^{\Dq(\om)}, \quad 
\Deqcn(\om) :=\overline{\Ciqcn(\om)}^{\Deq(\om)}.
\end{align}
For the full boundary case $ \Gamma_\tau = \Gamma $ (resp. $ \Gamma_\nu = \Gamma $) we set
$$\Dqc(\om) := \Dqct(\om),\quad \Deqc(\om) := \Deqcn(\om).$$
Furthermore, we define weak boundary conditions in the spaces 
\begin{align}
\label{defschwach}
\begin{split}
\cDqct(\om) &:=\bset{E\in\Dq(\om)}{\scpLtqom{E}{\cd \varphi} = -\scpLtqpoom{\ed E}{\varphi}~\text{for all}~ \varphi\in\Ciqpocn(\om)},\\
\cDeqcn(\om) &:=\bset{H\in\Deq(\om)}{\scpLtqom{H}{\ed\varphi} = -\scpLtqmoom{\cd H}{\varphi}~\text{for all}~ \varphi\in\Ciqmoct(\om)},
\end{split}
\end{align}
and again for $ \Gamma_\tau = \Gamma $ (resp. $ \Gamma_\nu = \Gamma $) we set
$$\cDqc(\om) := \cDqct(\om),\quad \cDeqc(\om) := \cDeqcn(\om).$$
We note that in definitions \eqref{defweakderi} and \eqref{defstark} the smooth test forms can by mollification be replaced by their respective Lipschitz continuous counterparts, e.g. $\Ciqct(\om)$ can be replaced by $\Czoqct(\om)$. Similarly, in definition \eqref{defschwach} the smooth test forms can by completion be replaced by their respective closures, i.e., $\Ciqpocn(\om)$ and $\Ciqmoct(\om)$  can be replaced by $\Deqpocn(\om)$ and $\Dqmoct(\om)$, respectively.
In \eqref{defstark} and \eqref{defschwach} homogeneous tangential and normal traces on $ \Gamma_\tau $, respectively $ \Gamma_\nu $, are generalised.
Clearly
$$\Dqct(\om)\subset\cDqct(\om),\quad\Deqcn(\om)\subset\cDeqcn(\om)
$$
and it will later be shown that in fact equality holds under our regularity assumptions on the boundary. In case of full boundary conditions the equality even holds without any assumptions on the regularity of the boundary, as can be seen by a short functional analytic argument, see \cite{bauerpaulyschomburgmaxcompweaklip}, but which is unavailable for the mixed boundary case.

We define the closed subspaces
\begin{align*}
\Dqz(\om):=\bset{E\in\Dq(\om)}{\ed E = 0},\quad \Deqz(\om):=\bset{E\in\Deq(\om)}{\cd E = 0}
\end{align*}
as well as $ \Dqczt(\om) := \Dqct(\om) \cap \Dqz(\om) $ and $\Deqczn(\om) := \Deqcn(\om) \cap \Deqz(\om) $.
Analogously for the weak spaces
\begin{align*}
\cDqczt(\om) := \cDqct(\om) \cap \Dqz(\om) ,\quad \cDeqczn(\om) := \cDeqcn(\om) \cap \Deqz(\om).
\end{align*}

In addition to the latter canonical Sobolev spaces
we will also need the classical Sobolev spaces for the Euclidean components of $q$-forms.
Note that $\om$, together with the global identity chart, is an $N$-dimensional Riemannian manifold.
In particular, $q$-forms $E\in\Ltqom$ can be represented  globally 
in Cartesian coordinates by their components $E_I$, i.e., $E = \sum_{I} E_I \mathrm{d} x^I$.
Here we use the ordered multi index notation $\mathrm{d}x^I = \mathrm{d}x^{i_1}\wedge 
\cdot\cdot\cdot\wedge\mathrm{d}x^{i_q}$ for $I=(i_1,...,i_q)\in\{1,...,N\}^q$.
The inner product for $E,H\in\Ltqom$ is given by 
\begin{align*}
\scpLtqom{E}{H} 
&= \int_{\om} E\wedge\star H 
= \sum_I \int_{\om} E_I H_I 
= \sum_I \scpLtom{E_I}{H_I} 
= \scpLtom{\vec E}{\vec H},
\end{align*}
where we introduce the vector proxy notation 
$$\vec E =[E_I]_I \in \Lt(\om;\rz^{N_q}),\quad N_q := \binom{N}{q}.$$
For $k\in\nz$ 
we can now define the Sobolev space $\Hkqom$ as the subset of $\Ltqom$ having each component $E_I$ in $\Hkom$.
In these cases, we have for $\norm{\alpha} \leq k$
$$\partial^\alpha E 
= \sum_I \partial^\alpha E_I \mathrm{d}x^I \quad\text{and}\quad \scp{E}{H}_{\Hkqom} 
:= \sum_{0\leq\norm{\alpha}\leq k}\scpLtqom{\partial^\alpha E}{\partial^\alpha H}$$
and we use the vector proxy notation also for the gradient, i.e.,
$$\na\vec E = [\partial_n E_I]_{n,I} = [...\na E_I...]_I \in \Lt(\om;\rz^{N\times N_q}).$$
In particular, for $E,H\in\Hoqom$
\begin{align*}
\scp{E}{H}_{\Hoqom}
&=\scpLtqom{E}{H}+\sum_{n=1}^N \scpLtqom{\partial_n E}{\partial_n H}
=\sum_I\big(\int_\om E_I H_I + \sum_n \int_\om \partial_n E_I \partial_n H_I\big)\\
&=\sum_I(\scpLtom{E_I}{H_I} + \scpLtom{\na E_I}{\na H_I}) 
= \scpLtom{\vec E}{\vec H} + \scpLtom{\na\vec E}{\na\vec H}
= \scp{\vec E}{\vec H}_{\Hoom}.
\end{align*}
Boundary conditions for $\Hoqom$-forms can again be defined strongly and weakly, i.e., by closure
\begin{align*}
\Hoqct(\om) := \overline{\Ciqc_{\Gamma_{\tau}}(\om)}^{\Hoqom}
\end{align*}
and by integration by parts
$$\cHoqct(\om) :=\bset{E\in\Hoqom}{\scpLtom{E_{I}}{\p_{n}\phi} = -\scpLtom{\p_{n}E_{I}}{\phi}
~\text{for all}~ n,I~\text{and all}~\phi\in\Cicn(\om)},$$
respectively. 
Le us also introduce the following Sobolev type spaces
\begin{align*}
\Dqk(\om):&=\bset{E\in\Hkq(\om)}{ \ed E \in\Hkqpo(\om)}, \\ 
\Deqk(\om):&=\bset{E\in\Hkq(\om)}{ \cd E \in\Hkqmo(\om)}.
\end{align*}

\begin{rem}
We emphasise that by switching $ \Gamma_\tau $ and $ \Gamma_\nu $ 
we can define the respective boundary conditions on the other part of the boundary as well.
Moreover, all definitions of our spaces extend literally
to any open subset $\om\subset\rN$ and any relatively open complementary boundary pairs $\Gamma_\tau$ and $\Gamma_\nu$.
\end{rem}

Finally we introduce our transformations $\eps$.

\begin{defi}
\label{defitrafo}
A transformation $\eps:\Ltqom\to\Ltqom$ will be called admissible,
if $\eps$ is bounded, symmetric, and uniformly positive definite.
More precisely, $\eps$ is a self-adjoint operator on $\Ltqom$ 
and there exists $\ul{\eps},\ol{\eps}>0$ such that for all $E\in\Ltqom$
$$\ul{\eps}\normLtqom{\eps E}\leq\normLtqom{E}\leq\ol{\eps}\sqrt{\scpLtqom{\eps E}{E}}.$$ 
\end{defi}

\subsection{Lipschitz Domains}

Let $\om\subset\rN$ be a bounded domain with boundary $\Gamma:=\p\!\om$.
We introduce the setting we will be working in. Define (cf. Figure \ref{fig:cube})
\begin{align*}
I&:=(-1,1),&
B&:=I^N\subset\rN, &B_{\pm}&:=\set{x\in B}{\pm x_N>0}, &
B_0&:=\set{x\in B}{x_N = 0}, \\
&&&&B_{0,\pm}&:=\set{x\in B_0}{\pm x_1 > 0 },&
B_{0,0}&:=\set{ x\in B_0 }{ x_1 = 0 }.
\end{align*}

\begin{defi}[weak Lipschitz domain]
\label{defilipmani}
$ \om $ is called weak Lipschitz, if the boundary $ \Gamma $ is a Lipschitz submanifold 
of the manifold $\ol{\om}$, i.e., 
there exist a finite open covering $ U_1,\dots, U_K\subset \rN $ of $ \Gamma $ 
and vector fields $ \phi_k : U_k \rightarrow B$, such that for $ k = 1,\dots, K $ 
\begin{itemize}
\item[\bf(i)] 
$ \phi_k \in \Czo(U_k,B) $ is bijective
and $ \psi_k := \phi_k^{-1} \in \Czo(B,U_k) $,
\item[\bf(ii)] 
$ \phi_k(U_k \cap \om) = B_{-}$
\end{itemize}
hold.
\end{defi}

\begin{rem}
For $ k = 1,\dots, K $ we have
$\phi_k (U_k \setminus \ol{\om}) = B_+ $ and $\phi_k (U_k \cap \Gamma) = B_0$.
\end{rem}

\begin{defi}[weak Lipschitz domain and weak Lipschitz interface]
\label{defilipsubmani}
Let $ \om $ be weak Lipschitz. A relatively open subset
$ \Gamma_\tau $ of $ \Gamma $ is called weak Lipschitz, 
if $ \Gamma_\tau $ is a Lipschitz submanifold of $ \Gamma $, i.e., 
there are an open covering $ U_1,\dots, U_K \subset \rN $ of $ \Gamma $ 
and vector fields $ \phi_k := U_k \rightarrow B $, 
such that for $ k = 1,\dots, K $ and in addition to (i), (ii) 
in Definition \ref{defilipmani} one of 
\begin{itemize}
\item[\bf(iii)] 
$ U_k \cap \Gamma_\tau = \emptyset$,
\item[\bf(iii$'$)] 
$ U_k \cap \Gamma_\tau = U_k \cap \Gamma \qimpl \phi_k(U_k \cap \Gamma_\tau) = B_{0} $,
\item[\bf(iii$''$)] 
$ \emptyset \neq U_k \cap \Gamma_\tau \neq U_k \cap \Gamma \qimpl \phi_k(U_k \cap \Gamma_\tau) = B_{0,-} $
\end{itemize}
holds. We define $ \Gamma_\nu := \Gamma \setminus \ol\Gamma_\tau $ 
to be the relatively open complement of $\Gamma_\tau$.
\end{defi}

\begin{defi}[weak Lipschitz pair]
A pair $ (\om,\Gamma_\tau) $ conforming to Definitions \ref{defilipmani} and
\ref{defilipsubmani} will be called weak Lipschitz.
\end{defi}

\begin{rem}
If $(\om,\Gamma_{\tau})$ is weak Lipschitz, so is $(\om,\Gamma_{\nu})$.
Moreover, for the cases (iii), (iii$'$) and (iii$''$)in Definition \ref{defilipsubmani} we further have
\begin{itemize}
\item[\bf(iii)] 
$ U_k \cap \Gamma_\tau = \emptyset \;\impl\; U_k \cap \Gamma_\nu = U_k \cap \Gamma \;\impl\; 
\phi_k(U_k \cap \Gamma_\nu) = B_{0}$,
\item[\bf(iii$'$)] 
$ U_k \cap \Gamma_\tau = U_k \cap \Gamma \;\impl\; U_k \cap \Gamma_\nu = \emptyset $,
\item[\bf(iii$''$)] 
$ \emptyset \neq U_k \cap \Gamma_\tau \neq U_k \cap \Gamma \;\impl \;
\emptyset \neq U_k \cap \Gamma_\nu \neq U_k \cap \Gamma \;\impl\; \phi_k(U_k \cap \Gamma_\nu) = B_{0,+} $
and $ \phi_k(U_k \cap \ol{\Gamma}_\tau \cap \ol{\Gamma}_\nu) = B_{0,0} $.
\end{itemize}
\end{rem}

In the literature the notion of a Lipschitz domain $\om\subset\rN$ 
is often used for a strong Lipschitz domain. For this let us define for $x\in\rN$
$$x':=(x_1,x_2,\dots,x_{N-1}),\quad
x'':=(x_2,\dots,x_{N-1}).$$

\begin{defi}[strong Lipschitz domain]
\label{defistrlip}
$ \om $ is called strong Lipschitz, 
if there are an open covering $ U_1,\dots, U_K \subset \rN $ of $\Gamma$,
rigid body motions $ R_k = A_k + a_k $, $ A_k\in\rNtN $ orthogonal, 
$ a_k \in \rN $, and $ \xi_k \in \Czo(I^{N-1},I)$, 
such that for $ k = 1,\dots, K $
\begin{itemize}
\item[\bf(i)]
$R_k(U_k \cap \om) = \bset{ x \in B }{ x_N < \xi_k(x') }$.
\end{itemize}
\end{defi}

\begin{rem}
For $ k = 1,\dots, K $ we have 
$$R_k(U_k \setminus \ol{\om}) = \bset{ x \in B}{x_N > \xi_k (x')},\quad
R_k(U_k \cap \Gamma) = \bset{ x \in B}{x_N = \xi_k (x')}.$$
\end{rem}

\begin{defi}[strong Lipschitz domain and strong Lipschitz interface]
\label{defistrlipinterface}
Let $ \om $ be strong Lipschitz. A relatively open subset
$ \Gamma_\tau $ of $ \Gamma $ is called strong Lipschitz, 
if there exist an open covering $ U_1,\dots, U_K \subset \rN $ of $ \Gamma $,
rigid body motions $R_k$,
and $ \xi_k \in \Czo(I^{N-1},I)$, $ \zeta_k \in \Czo(I^{N-2},I) $,
such that for $ k = 1,\dots, K $ and in addition to (i) 
in Definition \ref{defistrlip} one of 
\begin{itemize}
\item[\bf(ii)] 
$ U_k \cap \Gamma_\tau = \emptyset$,
\item[\bf(ii$'$)] 
$ U_k \cap \Gamma_\tau = U_k \cap \Gamma 
\qimpl 
R_k(U_k \cap \Gamma_{\tau}) = \bset{ x \in B}{x_N = \xi_k (x')}$,
\item[\bf(ii$''$)] 
$ \emptyset \neq U_k \cap \Gamma_\tau \neq U_k \cap \Gamma 
\qimpl
R_k(U_k \cap \Gamma_\tau) 
= \bset{ x\in B }{ x_N = \xi_k(x'),~ x_1 < \zeta_k (x'') }$
\end{itemize}
holds. We define $ \Gamma_\nu := \Gamma \setminus \ol\Gamma_\tau $ 
to be the relatively open complement of $\Gamma_\tau$.
\end{defi}

\begin{defi}[strong Lipschitz pair]
A pair $ (\om,\Gamma_\tau) $ conforming to Definitions \ref{defistrlip} and
\ref{defistrlipinterface} will be called strong Lipschitz.
\end{defi}

\begin{rem}
If $(\om,\Gamma_{\tau})$ is strong Lipschitz, so is $(\om,\Gamma_{\nu})$.
Moreover, for the cases (ii), (ii$'$) and (ii$''$) in Definition \ref{defistrlipinterface} 
we further have
\begin{itemize}
\item[\bf(ii)] 
$ U_k \cap \Gamma_\tau = \emptyset 
\;\impl\; 
U_k \cap \Gamma_\nu = U_k \cap \Gamma 
\;\impl\; 
R_k(U_k \cap \Gamma_{\nu}) = \bset{ x \in B}{x_N = \xi_k (x')}$,
\item[\bf(ii$'$)] 
$ U_k \cap \Gamma_\tau = U_k \cap \Gamma \;\impl\; U_k \cap \Gamma_\nu = \emptyset $,
\item[\bf(ii$''$)] 
$ \emptyset \neq U_k \cap \Gamma_\tau \neq U_k \cap \Gamma 
\;\impl \;
\emptyset \neq U_k \cap \Gamma_\nu \neq U_k \cap \Gamma 
\;\impl\; $
\begin{align*}
R_k(U_k \cap \Gamma_\nu) &= \bset{ x \in B }{ x_N = \xi_k (x'),~ x_1 > \zeta_k (x'') }, \\
R_k(U_k \cap \ol{\Gamma}_\tau \cap \ol{\Gamma}_\nu) &= \bset{ x \in B }{ x_N = \xi_k (x'),~ x_1 = \zeta_k (x'') }.
\end{align*}
\end{itemize}
\end{rem}

\begin{rem}
The following holds:
\begin{itemize}
\item[\bf(i)]  
$ \om $ strong Lipschitz $ \qimpl $ $ \om $ weak Lipschitz
\item[\bf(ii)] 
$ (\om,\Gamma_\tau) $ strong Lipschitz pair $ \qimpl $ $ (\om,\Gamma_\tau) $ weak Lipschitz pair
\end{itemize}
For a proof just define $\phi_{k}:=\varphi_{k}\circ R_{k}$ with $\varphi_{k}:U_{k}\to B$ given by
$$\varphi_{k}(x):=\dvec{x_{1}-\zeta_{k}(x'')}{x''}{x_{N}-\xi_{k}(x')}.$$
Note that the contrary does not hold as the implicit function theorem
is not available for Lipschitz maps.
\end{rem}

For later purposes we introduce special notations for the half-cube domain
\begin{align}
\label{halfcube}
\Xi:=B_{-},\quad \gamma:=\p\Xi
\end{align} 
and its relatively open boundary parts $\gamma_{\tau}$ and $\gamma_{\nu}:=\gamma\setminus\ol{\gamma_{\tau}}$.
We will only consider the cases 
\begin{align}
\label{halfcubegn}
\gamma_{\nu}=\emptyset,\quad\gamma_{\nu}=B_{0},\quad\gamma_{\nu}=B_{0,+}
\end{align}
and we note that $\Xi$ and $\gamma$, $\gamma_\tau$, $\gamma_\nu$ are strong Lipschitz, see Figure \ref{figureone}.

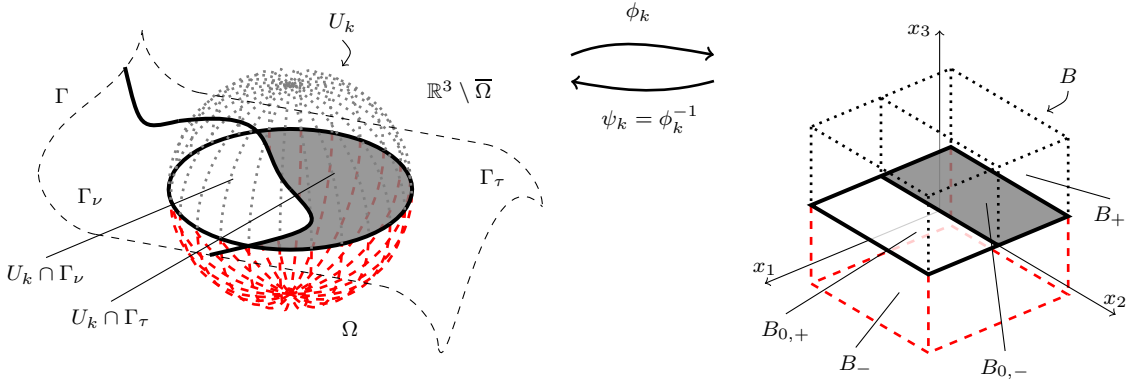
\begin{figure}
\centering
{\footnotesize		
\tdplotsetmaincoords{60}{20}
\begin{tikzpicture}[scale=0.4,tdplot_main_coords]
	\coordinate (O) at (0,0,0);		
	\def\Rad{4}				
	\foreach \angle in {0,15,...,165}{
		\tdplotsetthetaplanecoords{\angle}
		\tdplotdrawarc[red,dashed,line width=1pt,tdplot_rotated_coords]{(O)}{\Rad}{90}{270}{}{}
	}
	\fill [gray,opacity=.8] (O) circle (\Rad);
    	\fill [white] (0,0,0) -- (-1,1.3,0) -- (-2.35,3.3,0) -- (-3.3,2.3,0) -- (-3.95,0.9,0) -- (-3.95,-0.5,0) -- (-3.95,-0.9,0) -- (-3.3,-2.3,0) -- (-2.35,-3.3,0) -- (-1.1,-3.9,0) -- (0,-4,0) -- (1.3,-2.5,0) -- (1.4,-1.9,0) -- (1.1,-1.1,0) -- cycle;
	\draw [line width=1.5pt] (O) circle (\Rad);
	\foreach \angle in {-90,-75,...,75}{
		\tdplotsetthetaplanecoords{\angle}
		\tdplotdrawarc[gray,dotted,line width=1pt,tdplot_rotated_coords]{(O)}{\Rad}{-90}{90}{}{}
	}
	\draw [dashed] (-5,-5,0) -- (5,-5,0);
	\draw [dashed] (-5,-5,0) .. controls (-6,-5,0) and (-7,-5,1) .. (-7,-5,2);	
	\draw [dashed] (5,-5,0) .. controls (6,-5,0) and (7,-5,-1) .. (7,-5,-2);	
	\draw [dashed] (-5,5,0) -- (5,5,0);
	\draw [dashed] (-5,5,0) .. controls (-6,5,0) and (-7,5,1) .. (-7,5,2);	
	\draw [dashed] (5,5,0) .. controls (6,5,0) and (7,5,-1) .. (7,5,-2);		
	\draw [dashed] (-7,-5,2) .. controls (-7,-4,4) and (-7,4,0) .. (-7,5,2);	
	\draw [dashed] (7,-5,-2) .. controls (7,-4,-4) and (5,5,0) .. (7,5,-2);	
	\draw [line width=1.5pt] (-6.7,2.5,2) .. controls (-6,2,0) .. (-5,3,0);
	\draw [line width=1.5pt] (-5,3,0) .. controls (-2,5,0) and (-2,2,0) .. (0,0,0);
	\draw [line width=1.5pt] (0,0,0) .. controls (2,-2,0) .. (-1,-5,0);
	\draw (5,-8,0) 		node [] {$\Omega$};							
	\draw (3,8,0) 		node [] {$\mathbb{R}^3 \setminus \ol\Omega$};
	\draw (-8,0,2) 		node [] {$\Gamma$};							
	\draw (-6,-2.7,0) 	node [] {$\Gamma_\nu$};
	\draw (6,3,0) 		node [] {$\Gamma_\tau$};
	\draw [black,->] (0,5,3) node[anchor=south]{$U_k$} .. controls (0,6,2) .. (0,5,2);
	\draw [black] (-3,-9,0) node[anchor=north]{$U_k \cap \Gamma_\tau$} -- (1,1.5,0);
	\draw [black] (-6,-7,0) node[anchor=north]{$U_k \cap \Gamma_\nu$} -- (-2,0,0);
	\draw [black,line width=1pt,->] (8,5,4) .. controls (9.5,5,5) and (11,5,5) .. (13,5,5);
	\draw (10,6,5.5) node [] {$\phi_{k}$};
	\draw [black,line width=1pt,<-] (8,5,3) .. controls (9.5,5,3) and (11,5,3) .. (13,5,4);
	\draw (10.3,6.3,1.3) node [] {$\psi_{k} = \phi_{k}^{-1}$};
\end{tikzpicture}
\tdplotsetmaincoords{60}{140}
\begin{tikzpicture}[scale=1.2,tdplot_main_coords]
	\draw [black] (0,0,0) -- (1,0,0);
	\draw [black] (0,0,0) -- (0,1,0);
	\draw [gray] (0,0,0) -- (0,0,1);
	\draw [red,dashed,line width=1pt] (1,1,-1) -- (1,-1,-1) -- (-1,-1,-1) -- (-1,1,-1) -- cycle;	
	\draw [red,dashed,line width=1pt] (1,1,0) -- (1,1,-1);							
	\draw [red,dashed,line width=1pt] (1,-1,0) -- (1,-1,-1);							
	\draw [red,dashed,line width=1pt] (-1,-1,0) -- (-1,-1,-1);						
	\draw [red,dashed,line width=1pt] (-1,1,0) -- (-1,1,-1);							
	\fill [gray,opacity=.8] (0,1,0) -- (-1,1,0) -- (-1,-1,0) -- (0,-1,0) -- cycle;				
	\fill [white,opacity=.8] (0,1,0) -- (1,1,0) -- (1,-1,0) -- (0,-1,0) -- cycle;				
	\draw [black,line width=1.5pt] (0,-1,0) -- (0,1,0);								
	\draw [black,line width=1.5pt] (1,1,0) -- (1,-1,0) -- (-1,-1,0) -- (-1,1,0) -- cycle;		
	\draw [black,dotted,line width=1pt] (0,-1,1) -- (0,-1,0);					
	\draw [black,dotted,line width=1pt] (0,1,0) -- (0,1,1);					
	\draw [black,dotted,line width=1pt] (-1,-1,1) -- (-1,-1,0);				
	\draw [black,dotted,line width=1pt] (-1,1,1) -- (-1,1,0);					
	\draw [black,dotted,line width=1pt] (1,-1,1) -- (1,-1,0);					
	\draw [black,dotted,line width=1pt] (1,1,1) -- (1,1,0);					
	\draw [black,dotted,line width=1pt] (0,-1,1) -- (1,-1,1) -- (1,1,1) -- (0,1,1);			
	\draw [black,dotted,line width=1pt] (0,1,1) -- (-1,1,1) -- (-1,-1,1) -- (0,-1,1) -- cycle;	
	\draw [black,->] (1,0,0) -- (2.5,0,0) 	node[anchor=south]{$x_1$};
	\draw [black,->] (0,1,0) -- (0,3,0) 	node[anchor=south]{$x_2$};
	\draw [black,->] (0,0,1) -- (0,0,2.3) 	node[anchor=east]{$x_3$};
	\draw [black,->] (-1,1,1.6) 	node[anchor=south]{$B$} .. controls (-1.4,0.5,1.1) .. (-1.2,0.4,1);
	\draw [black] (2.2,0,-0.5) 	node[anchor=north]{$B_{0,+}$} -- (0.5,0.2,0);
	\draw [black] (2,3.5,0.5) 	node[anchor=north]{$B_{0,-}$} -- (-0.5,0.2,0);
	\draw [black] (1.2,0,-1.3) 	node[anchor=north]{$B_{-}$} -- (1.5,1.1,0);
	\draw [black] (-1.6,1.0,0) 	node[anchor=north]{$B_{+}$} -- (-0.9,0.4,0.3);
\end{tikzpicture}
}
\caption{Mappings $\phi_{k}$ and $\psi_{k}$ between a ball $U_k$ and the cube $B$.}
\label{figureone}
\end{figure}

\subsection{Outline of the Proof}

Let $(\om,\Gamma_{\tau})$ be a weak Lipschitz pair for a bounded domain $\om\subset\rN$. 
\begin{itemize}
\item 
As a first step, we observe $ \Hoqct(\om) = \cHoqct(\om) $, i.e., 
for the $ \Hoqom$-spaces the strong and weak definitions of the boundary conditions coincide,
see Lemma \ref{Hoct}.
\item 
In the second and essential step, we construct various regular $ \Hoq $-potentials on simple domains,
mainly for the half-cube $ \Xi$ from \eqref{halfcube} 
with the special boundary constellations \eqref{halfcubegn}, i.e.,
$$ \cDqczn(\Xi) = \Dqczgn(\Xi) = \ed\Hoqmocgn(\Xi), \quad 
\cDeqczn(\Xi) = \Deqczn(\Xi) = \cd\Hoqpocgn(\Xi),$$
see Section \ref{secregpot}.
Potentials of this type are called regular potentials.
\item 
In the third step, Section \ref{secstrweakXi}, 
it is shown that the strong and weak definitions of the boundary conditions coincide 
on the half-cube $\Xi$ from \eqref{halfcube} with the special boundary constellation \eqref{halfcubegn}, i.e.,
\begin{align}
\label{pillemann}
\cDqcn(\Xi) = \Dqcn(\Xi), \quad\cDeqcn(\Xi) = \Deqcn(\Xi). 
\end{align}
\item 
The fourth step proves the compact embedding on the half-cube $\Xi$ from \eqref{halfcube} 
with the special boundary constellations \eqref{halfcubegn}, i.e.,
\begin{align}
\label{MCPXi}
\Dqct(\Xi)\cap\eps^{-1}\Deqcn(\Xi)\hookrightarrow\Ltq(\Xi)
\end{align}
is compact, see Section \ref{seccompembhalfcube}.
\item 
In the fifth step, Theorem \ref{theoweakeqstrong}, \eqref{pillemann} is established for weak Lipschitz domains, i.e., 
$$\cDqct(\om) = \Dqct(\om),\quad\cDeqcn(\om) = \Deqcn(\om).$$
\item 
In the last step, we finally prove the compact embedding \eqref{MCPXi} for weak Lipschitz pairs, i.e.,
\begin{align*}
\Dqct(\om)\cap\eps^{-1}\Deqcn(\om)\hookrightarrow\Ltqom
\end{align*}
is compact, see our main result Theorem \ref{satzMKE}.
\end{itemize}

\subsection{Some Important Results}

Within our proofs we need a few important technical lemmas.
First, the strong and weak definitions of the boundary conditions coincide for $\Hoqom$-forms,
which is a density result for $\Hoqom$-forms.
This is an immediate consequence of the corresponding scalar result, 
whose proof can be found in \cite[Lemma 2, Lemma 3]{jochmanncompembmaxmixbc} and with a simplified proof in
\cite[Lemma 3.1]{bauerpaulyschomburgmaxcompweaklip}.

\begin{lem}[weak and strong boundary conditions coincide for $\Hoqom$]
\label{Hoct}
Let $\om\subset\rN$ be a bounded domain and let 
$(\om, \Gamma_{\tau})$ be a weak Lipschitz pair as well as
$$\Hoqcttr(\om):=\bset{u\in\Hoqom}{u|_{\Gamma_{\tau}}=0}$$
in the sense of traces. Then $\cHoqct(\om)=\Hoqcttr(\om)=\Hoqct(\om)$.
\end{lem}

Another crucial tool in our arguments is a universal extension operator for the Sobolev spaces 
$\Dqk(\om)$ and $\Deqk(\om)$ 
given in  \cite{hiptmairlizouunivextdiffforms}, which is based on the universal extension operator for standard Sobolev
spaces $\Hk(\om)$ introduced by Stein in \cite{steinsingintbook}.  
``Universality'' in this context means that the operator, which is given by a single formula, is able to extend all orders of 
Sobolev spaces simultaneously. More precisely, the following theorem, 
which is taken from \cite[Theorem 3.6]{hiptmairlizouunivextdiffforms}, holds:

\begin{lem}[Stein's extension operator]
\label{steinexop}
Let $\om\subset\rN$ be a bounded strong Lipschitz domain.
Then for $k\in\nz_0$ and $0\leq q\leq N$
there exists a (universal) linear and continuous extension operator
$$\E\,:\,\Dqk(\om)\rightarrow \Dqk(\rN).$$
More precisely, $\E$ satisfies $\E E = E$ a.e. in $\om$ and there exists $c>0$
such that for all $E\in\Dqk(\om)$
$$\normDqkrn{\E E} \leq c\normDqkom{E}.$$
Furthermore, $\E$ can be chosen such that $\E E$ has a fixed compact support in $\rN$
for all $E\in\Dqk(\om)$.
\end{lem}

Our third lemma
summarises well known and fundamental results 
for the theory of Maxwell's equations from \cite{picardpotential,picardcomimb}.
For this, we denote orthogonality and the orthogonal sum in $\Ltqom$
by $\bot$ and $\oplus$, respectively, and introduce
the harmonic Dirichlet and Neumann forms 
$$\mathcal{H}^{q}_{D}(\om):=\Dqcz(\om)\cap\Deqz(\om),\qquad
\mathcal{H}^{q}_{N}(\om):=\Dqz(\om)\cap\Deqcz(\om),$$
respectively.

\begin{lem}[Picard's generalisation of Weck's selection theorem, Helmholtz decompositions and Maxwell estimates]
\label{picardlem}
Let $\om\subset\rN$ be a bounded weak Lipschitz domain. Then the embeddings
\begin{align*}
\Dqcom\cap\Deqom\hookrightarrow\Ltqom,\qquad
\Dqom\cap\Deqcom\hookrightarrow\Ltqom
\end{align*}
are compact and $\mathcal{H}^{q}_{D}(\om)$, $\mathcal{H}^{q}_{N}(\om)$ are finite-dimensional.
Moreover, the Helmholtz decompositions 
\begin{align*}
\Ltqom&=\ed\Dqmocom\oplus\Deqzom
&
\Ltqom&=\ed\Dqmoom\oplus\Deqczom\\
&=\Dqczom\oplus\cd\Deqpoom
&
&=\Dqzom\oplus\cd\Deqpocom\\
&=\ed\Dqmocom\oplus\mathcal{H}^q_{D}(\om)\oplus\cd\Deqpo(\om),
&
&=\ed\Dqmoom\oplus\mathcal{H}^q_{N}(\om)\oplus\cd\Deqpoc(\om)
\end{align*}
are valid. In particular, all ranges are closed subspaces of $\Ltqom$ and 
\begin{align*}
\ed\Dqmocom
&=\Dqczom\cap\mathcal{H}^q_{D}(\om)^{\bot},
&
\ed\Dqmoom
&=\Dqzom\cap\mathcal{H}^q_{N}(\om)^{\bot},\\
\cd\Deqpoom
&=\Deqzom\cap\mathcal{H}^q_{D}(\om)^{\bot},
&
\cd\Deqpocom
&=\Deqczom\cap\mathcal{H}^q_{N}(\om)^{\bot}.
\end{align*}
Furthermore, there exists $c>0$
such that 
$$c\normLtqom{E}\leq\normLtqpoom{\ed E}+\normLtqmoom{\cd E}$$
holds for all $E\in\Dqcom\cap\Deqom\cap\mathcal{H}^q_{D}(\om)^{\bot}$
and all $E\in\Dqom\cap\Deqcom\cap\mathcal{H}^q_{N}(\om)^{\bot}$,
i.e., the Maxwell (or Friedrichs-Poincar\'e type) estimates are valid.
\end{lem}

\begin{cor}[refined Helmholtz decompositions]
\label{picardcorone}
Let $\om\subset\rN$ be a bounded weak Lipschitz domain.
Then
\begin{align*}
\Dqcom&=\Dqczom\oplus\big(\Dqcom\cap\cd\Deqpoom\big),
&
\ed\Dqcom&=\ed\big(\Dqcom\cap\cd\Deqpoom\big),\\
\Dqom&=\Dqzom\oplus\big(\Dqom\cap\cd\Deqpocom\big),
&
\ed\Dqom&=\ed\big(\Dqom\cap\cd\Deqpocom\big),\\
\Deqom&=\big(\ed\Dqmocom\cap\Deqom\big)\oplus\Deqzom,
&
\cd\Deqom&=\cd\big(\ed\Dqmocom\cap\Deqom\big),\\
\Deqcom&=\big(\ed\Dqmoom\cap\Deqcom\big)\oplus\Deqczom,
&
\cd\Deqcom&=\cd\big(\ed\Dqmoom\cap\Deqcom\big).
\end{align*}
\end{cor}

Let $\pi_{q,\om}:\Ltqom\to\cd\Deqpocom$ be the
orthonormal Helmholtz projector onto $\cd\Deqpocom$.
By the latter corollary $\pi_{q,\om}$ maps $\Dqom$ to 
$$\Dqom\cap\cd\Deqpocom=\Dqom\cap\Deqczom\cap\mathcal{H}^q_{N}(\om)^{\bot}.$$

\begin{cor}[Maxwell estimate for $\ed$ and Neumann boundary condition]
\label{picardcortwo}
Let $\om\subset\rN$ be a bounded weak Lipschitz domain.
Then for all $E\in\Dqom$ it holds $\pi_{q,\om}E\in\Dqom\cap\cd\Deqpocom$
and $\ed\pi_{q,\om}E=\ed E$ as well as 
$$c\normLtqom{\pi_{q,\om}E}\leq\normLtqpoom{\ed E},$$
with $c$ from Lemma \ref{picardlem}. 
\end{cor}

If $\om=\rN$ a similar theory holds true utilising polynomially weighted Sobolev spaces,
see \cite{picardpotential} for details. Let $\pi_{q,\rN}:\Ltq(\rN)\to\Deqz(\rN)$
be the orthonormal Helmholtz projector onto $\Deqz(\rN)$.

\begin{lem}[Helmholtz decompositions and Maxwell estimate for $\ed$ in the whole space]
\label{picardlemrN}
It holds $\mathcal{H}^q_{N}(\rN)=\mathcal{H}^q_{D}(\rN)=\{0\}$ and 
$$\Ltq(\rN)=\Dqz(\rN)\oplus\Deqz(\rN),\quad
\Dq(\rN)=\Dqz(\rN)\oplus\big(\Dq(\rN)\cap\Deqz(\rN)\big).$$
Moreover, for all $E\in\Dq(\rN)$ it holds $\pi_{q,\rN}E\in\Dq(\rN)\cap\Deqz(\rN)$
and $\ed\pi_{q,\rN}E=\ed E$ as well as 
$$\norm{\pi_{q,\rN}E}_{\Dq(\rN)}\leq\norm{E}_{\Dq(\rN)}.$$
\end{lem}

Regularity in the whole space, see e.g. \cite[(4.7) or Lemma 4.2 (i)]{kuhnpaulyregmax}, 
shows the following result.

\begin{lem}[regularity in the whole space]
\label{kuhnpaulylemrN}
$\Dq(\rN)\cap\Deq(\rN)=\Hoq(\rN)$ with equal norms.
More precisely, $E\in\Dq(\rN)\cap\Deq(\rN)$ if and only if $E\in\Hoq(\rN)$ and 
$$\norm{E}_{\Hoq(\rN)}^2
=\norm{E}_{\Ltq(\rN)}^2+\norm{\ed E}_{\Ltqpo(\rN)}^2+\norm{\cd E}_{\Ltqmo(\rN)}^2.$$
\end{lem}

\section{Regular Potentials}
\label{secregpot}

As one of our main steps (step 4), in Section \ref{seccompembhalfcube} 
the compact embedding is proved on the half-cube $\Xi\subset\rN$.
This will be achieved (in step 2) by constructing regular $\Ho(\Xi)$-potentials 
for $ \ed $-free and $ \cd $-free $\Ltq(\Xi)$-forms,
which will then enable us to use Rellich's selection theorem. 
This section is devoted to the construction and existence of these regular potentials, i.e.,
to step 2.

\subsection{Regular Potentials Without Boundary Conditions}

Let us recall
$$\ed\Dqmoom=\Dqzom\cap\mathcal{H}^q_{N}(\om)^{\bot},\quad
\cd\Deqpoom=\Deqzom\cap\mathcal{H}^q_{D}(\om)^{\bot}$$
from Lemma \ref{picardlem}.
The next two lemmas ensure the existence of $ \Hoqom $-potentials without boundary conditions
for strong Lipschitz domains. 

\begin{lem}[regular potential for $\ed$ without boundary condition]
	\label{Dqzpot}
	Let $\om\subset\rN$ be a bounded strong Lipschitz domain.
	Then there exists a continuous linear operator
	$$\T_{\ed}:\Dqzom\cap\mathcal{H}^{q}_{N}(\om)^{\perp}\rightarrow\Hoqmo(\rN)\cap\Deqmoz(\rN)$$ 
	such that for all $E\in\Dqzom\cap\mathcal{H}^{q}_{N}(\om)^{\perp}$ 
	$$\ed\T_{\ed} E=E\quad\text{in }\om.$$
	Especially
	$$\Dqzom\cap\mathcal{H}^{q}_{N}(\om)^{\perp}
	=\ed\Hoqmoom=\ed\big(\Hoqmoom\cap\Deqmozom\big)$$ 
	and the regular potential 
	depends continuously on the data.
	Particularly, these are closed subspaces of $\Ltqom$ and $\T_{\ed}$ is a right inverse to $\ed$.
	By a simple cut-off technique $\T_{\ed}$ may be modified to 
	$$\T_{\ed}:\Dqzom\cap\mathcal{H}^{q}_{N}(\om)^{\perp}\rightarrow\Hoqmo(\rN)$$ 
	such that $\T_{\ed}E$ has a fixed compact support in $\rN$ 
	for all $E\in\Dqzom\cap\mathcal{H}^{q}_{N}(\om)^{\perp}$.
\end{lem}

\begin{proof}
Suppose $ E\in\Dqzom\cap\mathcal{H}^{q}_{N}(\om)^{\perp} $.
By Lemma \ref{picardlem} there exists $ H\in\Dqmoom $ with $ \ed H = E $ in $ \om $. 
Applying Corollary \ref{picardcortwo} we get
$\pi_{q-1,\om}H\in\Dqmoom\cap\cd\Deqcom $ 
with $\ed\pi_{q-1,\om}H=\ed H=E$ and 
$$\norm{\pi_{q-1,\om} H}_{\Dqmoom}\leq c\normLtqom{E}.$$
Note that $\pi_{q-1,\om} H$ is uniquely determined.
By the Stein extension operator $\E:\Dgen{0,q-1}{}{}(\om)\rightarrow \Dgen{0,q-1}{}{}(\rN)$
from Lemma \ref{steinexop} we have
$\E\pi_{q-1,\om}H\in\Dgen{0,q-1}{}{}(\rN)$ with compact support.
Projecting again, now with Lemma \ref{picardlemrN} onto $ \Deqmoz(\rN) $, we obtain 
$\pi_{q-1,\rN}\E\pi_{q-1,\om}H\in\Dqmo(\rN)\cap\Deqmoz(\rN)$
(again uniquely determined) with $\ed\pi_{q-1,\rN}\E\pi_{q-1,\om}H=\ed\E\pi_{q-1,\om}H$ and
$$\norm{\pi_{q-1,\rN}\E\pi_{q-1,\om}H}_{\Dqmo(\rN)}
\leq\norm{\E\pi_{q-1,\om}H}_{\Dqmo(\rN)}
\leq c\norm{\pi_{q-1,\om}H}_{\Dqmo(\om)}.$$
Lemma \ref{kuhnpaulylemrN} shows $\pi_{q-1,\rN}\E\pi_{q-1,\om}H\in\Hoqmo(\rN)\cap\Deqmoz(\rN)$ with
$$\norm{\pi_{q-1,\rN}\E\pi_{q-1,\om}H}_{\Hoqmo(\rN)}
=\norm{\pi_{q-1,\rN}\E\pi_{q-1,\om}H}_{\Dqmo(\rN)}.$$
Finally, 
$\T_{\ed}E:=\pi_{q-1,\rN}\E\pi_{q-1,\om}H\in\Hoqmo(\rN)\cap\Deqmoz(\rN)$ 
meets our needs as
$$\norm{\T_{\ed}E}_{\Hoqmo(\rN)}
\leq c\normLtqom{E}$$
and $\ed\T_{\ed}E=\ed\pi_{q-1,\rN}\E\pi_{q-1,\om}H=\ed\E\pi_{q-1,\om}H=\ed\pi_{q-1,\om}H=\ed H=E$ in $\om$.
\end{proof}

By Hodge-$\star$-duality we get a corresponding result for the $ \cd $-operator.

\begin{lem}[regular potential for $\cd$ without boundary condition]
	\label{Deqzpot}
	Let $\om\subset\rN$ be a bounded strong Lipschitz domain.
	Then there exists a continuous linear operator
	$$\T_{\cd}:\Deqzom\cap\mathcal{H}_{D}^{q}(\om)^\perp\rightarrow\Hoqpo(\rN)\cap\Dqpoz(\rN),$$ 
	such that for all $E\in\Deqzom\cap\mathcal{H}_{D}^{q}(\om)^\perp$ 
	$$\cd\T_{\cd}E=E\quad\text{in }\om.$$
	Especially
	$$\Deqzom\cap\mathcal{H}_{D}^{q}(\om)^\perp
	=\cd\Hoqpoom=\cd\big(\Hoqpoom\cap\Dqpozom\big)$$ 
	and the regular potential depends 		continuously on the data.
	In particular, these are closed subspaces of $\Ltqom$ and $\T_{\cd}$ is a right inverse to $\cd$.
    By a simple cut-off technique $\T_{\cd}$ may be modified to 
	$$\T_{\cd}:\Deqzom\cap\mathcal{H}^{q}_{D}(\om)^{\perp}\rightarrow\Hoqpo(\rN)$$ 
	such that $\T_{\cd}E$ has a fixed compact support in $\rN$ 
	for all $E\in\Deqzom\cap\mathcal{H}^{q}_{D}(\om)^{\perp}$.
\end{lem}

\subsection{Regular Potentials With Boundary Conditions for the Half-Cube}

Now we start constructing $\Hoq(\Xi)$-potentials on $\Xi$ with boundary conditions. 
Let us recall our special setting on the half-cube
$$\Xi=B_{-}\quad\text{and}\quad
\gamma_{\nu}=\emptyset,\quad\gamma_{\nu}=B_{0}\quad\text{or}\quad\gamma_{\nu}=B_{0,+}.$$
Furthermore, cf. Figure \ref{fig:cube}, we extend $ \Xi $ over $ \gamma_\nu $ by
\begin{align*}
\widetilde \Xi&\,=
\text{\rm int}(\ol\Xi\cup\ol{\widehat \Xi}),&
\widehat \Xi&:=
\begin{cases}
\set{x\in B}{x_N>0}=B_{+}
&\text{, if }\gamma_{\nu}=B_{0},\\
\set{x\in B}{x_N,x_1>0}
=\set{x\in B_{+}}{x_1>0}=:B_{+,+}
&\text{, if }\gamma_{\nu}=B_{0,+}.
\end{cases}
\end{align*}

\begin{lem}[regular potential for $\ed$ with partial boundary condition on the half-cube]
\label{satzD0L6}
There exists a continuous linear operator
$$\S_{\ed}:\cDqczgn(\Xi)\rightarrow\Hoqmo(\rN)\cap\Hoqmocgn(\Xi),$$
such that for all $H\in\cDqczgn(\Xi)$
$$\ed\S_{\ed} H=H\quad\text{in }\Xi.$$
Especially 
$$\cDqczgn(\Xi)=\Dqczgn (\Xi)=\ed\Hoqmocgn(\Xi)=\ed\Dqmocgn(\Xi)=\ed\cDqmocgn(\Xi)$$
and the regular $\Hoqmocgn(\Xi)$-potential depends continuously on the data.
In particular, these spaces are closed subspaces of $\Ltq(\Xi)$ and $\S_{\ed}$ is a right inverse to $\ed$.
Without loss of generality, 
$\S_{\ed}$ maps to forms with a fixed compact support in $\rN$.
\end{lem}

\begin{proof}
The case $\gamma_{\nu}=\emptyset$ is done in Lemma \ref{Dqzpot}. 
Hence let $\gamma_{\nu}=B_{0}$ or $\gamma_{\nu}=B_{0,+}$.
Suppose $H\in\cDqczgn(\Xi)$ and define $\widetilde H\in\Ltq(\widetilde \Xi)$
as extension of $H$ by zero to $\widehat{\Xi}$ by
\begin{align}
\label{defHschlange}
\widetilde H:=
\begin{cases}
H&\text{in }\Xi,\\
0&\text{in }\widehat \Xi.
\end{cases}
\end{align}
By definition of $\cDqczgn(\Xi)$ (weak boundary condition) 
it follows $\ed \widetilde H=0$ in $\widetilde \Xi$, i.e., $\widetilde H \in\Dqz(\widetilde \Xi)$.
Because $\widetilde \Xi$ is strong Lipschitz 
and topologically trivial, especially $\mathcal{H}^{q}_{N}(\widetilde{\Xi})=\{0\}$,
Lemma \ref{Dqzpot} yields a regular potential 
$ E = \T_{\ed}\widetilde H\in\Hoqmo(\rN)\cap\Dqmoz(\rN)$
with $\ed E=\widetilde H$ in $\widetilde \Xi$ and 
$$\norm{E}_{\Hoqmo(\rN)}
\leq c\norm{\widetilde{H}}_{\Ltq(\widetilde{\Xi})}
\leq c\norm{H}_{\Ltq(\Xi)}.$$
In particular, $E\in\Hoqmo(\widehat \Xi)$ and $\ed E = 0$ in $\widehat \Xi$,
i.e., $E\in\Hoqmo(\widehat\Xi)\cap\Dqmoz(\widehat\Xi)$.
Using Lemma \ref{Dqzpot} again, this time in $\widehat{\Xi}$,
we obtain $F=\T_{\ed}E\in\Hoqmt(\rN)\subset\Hoqmt(\widehat \Xi)$ with
$\ed F = E \text{ in }\widehat \Xi$ and 
$$\norm{F}_{\Hoqmt(\rN)}
\leq c\norm{E}_{\Ltq(\widehat{\Xi})}.$$
Since $ E \in \Hoqmo(\widehat \Xi) $ we have $F\in \Dqomt(\widehat \Xi)$.
Let $\E:\Dqomt(\widehat \Xi)\rightarrow\Dqomt(\rN)$ be the Stein extension operator
from Lemma \ref{steinexop}. Then
\begin{align*}
\Abb{\S_{\ed}}{\cDqczgn(\Xi)}{\Hoqmo(\rN)}{H}{E -\ed(\E F)}
\end{align*}
is linear and continuous as
\begin{align*}
\norm{\S_{\ed}H}_{\Hoqmo(\rN)}
&\leq\norm{E}_{\Hoqmo(\rN)}
+\norm{\E F}_{\Dqomt(\rN)}\\
&\leq\norm{E}_{\Hoqmo(\rN)}
+\norm{F}_{\Dqomt(\widehat\Xi)}
\leq\norm{E}_{\Hoqmo(\rN)}
\leq c\norm{H}_{\Ltq(\Xi)}.
\end{align*}
Since $\S_{\ed} H = 0$ in $\widehat \Xi$, we have $\S_{\ed} H|_{\gamma_{\nu}} = 0$, 
which means $\S_{\ed} H\in\Hoqmocgntr(\Xi)$.
Therefore, by Lemma \ref{Hoct} we see 
$\S_{\ed} H \in\Hoqmocgn(\Xi)\subset\Dqmocgn(\Xi)\subset\cDqmocgn(\Xi)$.
Moreover, $\ed(\S_{\ed} H)=\ed E=\widetilde{H}$ in $\widetilde\Xi$,
especially $\ed(\S_{\ed} H)=H$ in $\Xi$.
Finally we note
$$\ed\Hoqmocgn(\Xi)
\subset\ed\Dqmocgn(\Xi)
\subset\Dqczgn(\Xi),\,
\ed\cDqmocgn(\Xi)
\subset\cDqczgn(\Xi)
\subset\ed\Hoqmocgn(\Xi),$$
completing the proof.
\end{proof}

Again by Hodge-$ \star $-duality we obtain the following.

\begin{lem}[regular potential for $\cd$ with partial boundary condition on the half-cube]
\label{satzLtL6}
There exists a continuous linear operator
$$\S_{\cd}:\cDeqczgn(\Xi)\rightarrow\Hoqpo(\rN)\cap\Hoqpocgn(\Xi),$$
such that for all $H\in\cDeqczgn(\Xi)$
$$\cd\S_{\cd} H=H\quad\text{in }\Xi.$$
Especially 
$$\cDeqczgn(\Xi)=\Deqczgn (\Xi)=\cd\Hoqpocgn(\Xi)=\cd\Deqpocgn(\Xi)=\cd\cDeqpocgn(\Xi)$$
and the regular $\Hoqpocgn(\Xi)$-potential depends continuously on the data.
In particular, these spaces are closed subspaces of $\Ltq(\Xi)$ and $\S_{\cd}$ is a right inverse to $\cd$.
Without loss of generality, 
$\S_{\cd}$ maps to forms with a fixed compact support in $\rN$.
\end{lem}

\begin{figure}
\centering
{\footnotesize		
\tdplotsetmaincoords{60}{140}
\begin{tikzpicture}[scale=1.6,tdplot_main_coords]
	\draw [black] (0,0,0) -- (1,0,0);
	\draw [black] (0,0,0) -- (0,1,0);
	\draw [gray] (0,0,0) -- (0,0,1.4);
	\draw [red,dashed,line width=1pt] (1,1,-1) -- (1,-1,-1) -- (-1,-1,-1) -- (-1,1,-1) -- cycle;	
	\draw [red,dashed,line width=1pt] (1,1,0) -- (1,1,-1);							
	\draw [red,dashed,line width=1pt] (1,-1,0) -- (1,-1,-1);							
	\draw [red,dashed,line width=1pt] (-1,-1,0) -- (-1,-1,-1);						
	\draw [red,dashed,line width=1pt] (-1,1,0) -- (-1,1,-1);							
	\fill [white,opacity=.7] (1,1,0) -- (1,-1,0) -- (-1,-1,0) -- (-1,1,0) -- cycle;				
	\draw [black,line width=1.5pt] (1,1,0) -- (1,-1,0) -- (-1,-1,0) -- (-1,1,0) -- cycle;		
	\draw [black,dotted,line width=1pt] (1,-1,1) -- (1,-1,0);							
	\draw [black,dotted,line width=1pt] (1,1,1) -- (1,1,0);							
	\draw [black,dotted,line width=1pt] (-1,-1,1) -- (-1,-1,0);						
	\draw [black,dotted,line width=1pt] (-1,1,1) -- (-1,1,0);							
	\fill [white,opacity=.5] (-1,-1,1) -- (1,-1,1) -- (1,1,1) -- (-1,1,1) -- cycle;				
	\draw [black,dotted,line width=1pt] (-1,-1,1) -- (1,-1,1) -- (1,1,1) -- (-1,1,1) -- cycle;	
	\draw [black,->] (1,0,0) -- (2.5,0,0) 	node[anchor=south]{$x_1$};
	\draw [black,->] (0,1,0) -- (0,3,0) 	node[anchor=south]{$x_2$};
	\draw [black,->] (0,0,1) -- (0,0,2.3) 	node[anchor=east]{$x_3$};
  \draw [color=black,->] (3,1,0) -- (1.5,1,0);
  \draw [color=black,->] (3.1,1,2) -- (1.5,1,1);
  \draw [color=black,->] (0.7,2,3) -- (1.2,1,2);
	\draw (0.7,2.4,3.2) node [] {\Large$\widehat{\Xi}$};
  \draw (3.4,1.1,0.2) node [] {\Large$\Xi$};
  \draw (3.4,1,2.1) node [] {\Large$B_{0}$};
\end{tikzpicture}
\hspace*{5mm}
\tdplotsetmaincoords{60}{140}
\begin{tikzpicture}[scale=1.6,tdplot_main_coords]
	\draw [black] (0,0,0) -- (1,0,0);
	\draw [black] (0,0,0) -- (0,1,0);
	\draw [black] (0,0,0) -- (0,0,1);
	\draw [red,dashed,line width=1pt] (1,1,-1) -- (1,-1,-1) -- (-1,-1,-1) -- (-1,1,-1) -- cycle;	
	\draw [red,dashed,line width=1pt] (1,1,0) -- (1,1,-1);							
	\draw [red,dashed,line width=1pt] (1,-1,0) -- (1,-1,-1);							
	\draw [red,dashed,line width=1pt] (-1,-1,0) -- (-1,-1,-1);						
	\draw [red,dashed,line width=1pt] (-1,1,0) -- (-1,1,-1);							
	\fill [white,opacity=.8] (-1,-1,0) -- (1,-1,0) -- (1,1,0) -- (-1,1,0) -- cycle;				
	\draw [red,dashed,line width=1pt] (0,-1,0) -- (0,1,0) -- (-1,1,0) -- (-1,-1,0) -- cycle; 	
	\draw [black,line width=1.5pt] (1,1,0) -- (1,-1,0) -- (0,-1,0) -- (0,1,0) -- cycle;		
	\draw [black,dotted,line width=1pt] (0,-1,1) -- (0,-1,0);							
	\draw [black,dotted,line width=1pt] (1,-1,1) -- (1,-1,0);							
	\draw [black,dotted,line width=1pt] (0,1,0) -- (0,1,1);							
	\draw [black,dotted,line width=1pt] (1,1,1) -- (1,1,0);							
	\fill [white,opacity=.8] (0,-1,1) -- (1,-1,1) -- (1,1,1) -- (0,1,1) -- cycle;				
	\draw [black,dotted,line width=1pt] (0,-1,1) -- (1,-1,1) -- (1,1,1) -- (0,1,1) -- cycle;	
	\draw [black,->] (1,0,0) -- (2.5,0,0) node[anchor=south]{$x_1$};
	\draw [black,->] (0,1,0) -- (0,3,0) node[anchor=south]{$x_2$};
	\draw [black,->] (0,0,1) -- (0,0,2) node[anchor=east]{$x_3$};
  \draw [color=black,->] (3,1,0) -- (1.5,1,0);
  \draw [color=black,->] (1,2,2.5) -- (1,1,1.4);
  \draw [color=black,->] (3.1,1,2) -- (1.5,1,1);
	\draw (1,2.4,2.8) node [] {\Large$\widehat{\Xi}$};
  \draw (3.4,1.1,0.2) node [] {\Large$\Xi$};
  \draw (3.5,1,2.1) node [] {\Large$B_{0,+}$};
\end{tikzpicture}
}
\caption{The half-cube $\Xi=B_{-}$, extended by $\widehat\Xi$ to the polygonal domain $\widetilde\Xi$,
and the rectangles $\gamma_{\nu}=B_{0}$ and $\gamma_{\nu}=B_{0,+}$.} 
\label{fig:cube}
\end{figure}
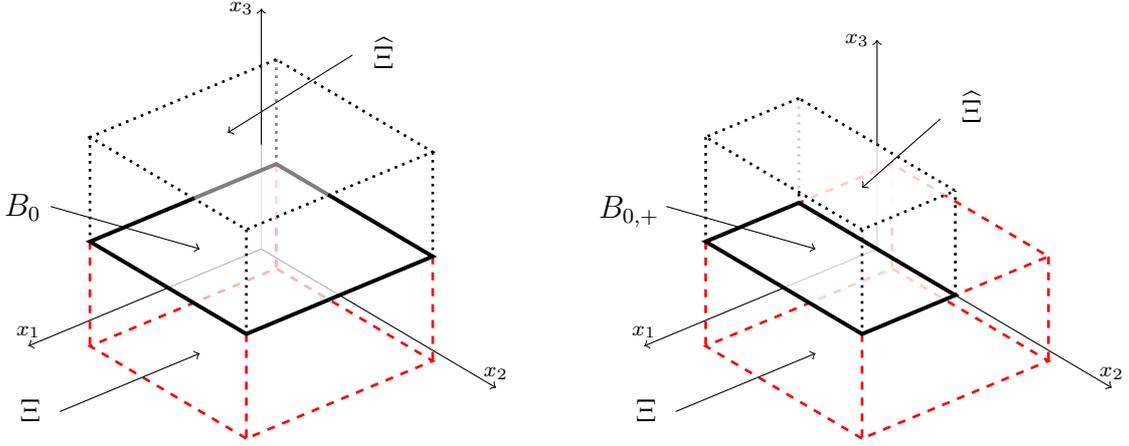

\subsection{Weak and Strong Boundary Conditions Coincide for the Half-Cube}
\label{secstrweakXi}

Now the two main density results immediately follow.
We note that this has already been proved for the $\Hoqom$-spaces in Lemma \ref{Hoct}, i.e., $\cHoqct(\om)=\Hoqct(\om)$.

\begin{lem}[weak and strong boundary conditions coincide for the half-cube]
\label{rss}
$\cDqcgn(\Xi)=\Dqcgn(\Xi)$ and $\cDeqcgn(\Xi) =\Deqcgn(\Xi)$.
\end{lem}

\begin{proof}
Suppose $E\in\cDqcgn(\Xi)$ and thus $\ed E\in\cDqpoczgn(\Xi)$. By Lemma \ref{satzD0L6} there exists 
$H = \S_{\ed} \ed E\in\Hoqcgn(\Xi)$ with $\ed H= \ed E$. 
By Lemma \ref{satzD0L6} we get $E-H\in\cDqczgn(\Xi)=\Dqczgn(\Xi)$
and hence $E\in\Dqcgn(\Xi)$.
\end{proof}

\section{Weck's Selection Theorem}
\label{seccompemb}

\subsection{The Compact Embedding for the Half-Cube}
\label{seccompembhalfcube}

First we show the main result on the half-cube $\Xi=B_{-}$ with the special boundary patches
$$\gamma_{\nu}=\emptyset,\quad\gamma_{\nu}=B_{0}\quad\text{or}\quad\gamma_{\nu}=B_{0,+}$$
from the latter section. 
To this end let $\eps$ be an admissible transformation on $\Ltq(\Xi)$
and let us consider the densely defined and closed (unbounded) linear operator
$$\ed_{\tau}^{q-1}:\Dqmocgt(\Xi)\subset\Ltqmo(\Xi)\rightarrow\Ltqeps(\Xi)\,;\quad
E\mapsto\ed E$$
together with its (Hilbert space) adjoint
$$-\cd^q_{\nu}\eps:=(\ed^{q-1}_{\tau})^*:\eps^{-1}\cDeqcgn(\Xi)\subset\Ltqeps(\Xi)\rightarrow\Ltqmo(\Xi)\,;\quad
H\mapsto-\cd\eps H.$$
Note that by Lemma \ref{rss} we have $\cDeqcgn(\Xi)=\Deqcgn(\Xi)$. 
Here, $\Ltqeps(\Xi)$ denotes $\Ltq(\Xi)$ equipped with the inner product
$\scp{\,\cdot\,}{\,\cdot\,}_{\LtqepsXi} := \scpLtqXi{\eps\,\cdot\,}{\,\cdot\,}$.
Let $\oplus_{\eps}$ denote the orthogonal sum with respect to the $\Ltqeps$-scalar product. 
The projection theorem yields immediately:

\begin{lem}[regular Helmholtz decompositions for the half-cube]
\label{lemHDXi}
The Helmholtz decompositions
\begin{align*}
\Ltqeps(\Xi) 
&= \Dqczgt (\Xi)\oplus_{\eps}\eps^{-1}\Deqczgn(\Xi),
&
\Dqczgt (\Xi)&=\ed\Hoqmocgt(\Xi),
&
\Deqczgn(\Xi)&=\cd\Hoqpocgn(\Xi)
\end{align*}
hold. Moreover, the refined Helmholtz decompositions
\begin{align*}
\Dqcgt(\Xi)
&=\ed\Hoqmocgt(\Xi)\oplus_{\eps}\big(\Dqcgt(\Xi)\cap\eps^{-1}\cd\Hoqpocgn(\Xi)\big),\\
\eps^{-1}\Deqcgn(\Xi)
&=\big(\ed\Hoqmocgt(\Xi)\cap\eps^{-1}\Deqcgn(\Xi)\big)\oplus_{\eps}\eps^{-1}\cd\Hoqpocgn(\Xi),\\
\Dqcgt(\Xi)\cap\eps^{-1}\Deqcgn(\Xi)
&=\big(\ed\Hoqmocgt(\Xi)\cap\eps^{-1}\Deqcgn(\Xi)\big)\oplus_{\eps}\big(\Dqcgt(\Xi)\cap\eps^{-1}\cd\Hoqpocgn(\Xi)\big)
\end{align*}
are valid, and the respective regular potentials,
given by the operators $\S_{\ed}$ and $\S_{\cd}$
from Lemma \ref{satzD0L6} and Lemma \ref{satzLtL6}, respectively,
depend continuously on the data.
\end{lem}

\begin{proof}
The projection theorem yields
$\Ltqeps(\Xi)=\overline{\ed\Dqmocgt(\Xi)}\oplus_{\eps}\eps^{-1}\cDeqczgn(\Xi)$.
Furthermore, 
$$\overline{\ed\Dqmocgt(\Xi)} = \ed\Dqmocgt(\Xi) = \ed\Hoqmot(\Xi)=\Dqczgt (\Xi)$$ 
by Lemma \ref{satzD0L6} and 
$$\cDeqczgn(\Xi) =\Deqczgn(\Xi)=\cd\Hoqpon(\Xi)$$ 
by Lemma \ref{satzLtL6}. 
The other assertions follow immediately.
\end{proof}

\begin{lem}[Weck's selection theorem for the half-cube]
\label{HSG}
The embedding $\Dqcgt(\Xi)\cap\eps^{-1}\Deqcgn(\Xi) \hookrightarrow \Ltqeps(\Xi)$ is compact.
\end{lem}

\begin{proof}
Let $(H_n)_{n\in\nz}$ be a bounded sequence in $\Dqcgt(\Xi)\cap\eps^{-1}\Deqcgn(\Xi)$. 
By Lemma \ref{lemHDXi} 
we can decompose
\begin{align*}
H_n = H_n^{\ed} + H_n^{\cd} = \ed E_n^{\ed} + \eps^{-1}\cd E_n^{\cd} 
\in \big(\ed\Hoqmocgt(\Xi) \cap \eps^{-1}\Deqcgn(\Xi)\big) 
\oplus_{\eps}\big( \Dqcgt(\Xi)\cap\eps^{-1}\cd\Hoqpocgn(\Xi)\big),
\end{align*}
with $E_n^{\ed}=\S_{\ed}H_n^{\ed}$ and $E_n^{\cd}=\S_{\cd}H_n^{\cd}$.
Then $\ed H^{\cd}_{n} = \ed H_{n}$ and $\cd\eps H^{\ed}_{n} = \cd\eps H_{n}$
as well as
\begin{align*}
\normHoqmoxi{E_n^{\ed}}\leq c\,\normLtqXi{H_n^{\ed}}\leq c\,\normLtqepsXi{H_n},\\
\normHoqpoxi{E_n^{\cd}}\leq c\,\normLtqXi{H_n^{\cd}}\leq c\,\normLtqepsXi{H_n}.
\end{align*}
By Rellich's selection theorem and without loss of generality
$ (E_n^{\ed}) $ and $ (E_n^{\cd}) $ converge in $ \Ltqmo(\Xi) $ and $ \Ltqpo(\Xi) $, respectively.
Moreover, 
\begin{align*}
\normLtqepsXi{H_n^{\ed} - H_m^{\ed}}^2
&=\scpLtqepsXi{H_n^{\ed} - H_m^{\ed}}{\ed(E_n^{\ed}-E_m^{\ed})}\\
&=-\scpLtqmoXi{\cd\eps(H_n^{\ed} - H_m^{\ed})}{E_n^{\ed}-E_m^{\ed}}
\leq c\,\normLtqmoXi{E_n^{\ed}-E_m^{\ed}},\\
\normLtqepsXi{H_n^{\cd} - H_m^{\cd}}^2
&=\scpLtqepsXi{H_n^{\cd} - H_m^{\cd}}{\eps^{-1}\cd(E_n^{\cd}-E_m^{\cd})}\\
&=-\scp{\ed(H_n^{\cd} - H_m^{\cd})}{E_n^{\cd}-E_m^{\cd}}_{\Ltqpo(\Xi)}
\leq c\,\norm{E_n^{\cd}-E_m^{\cd}}_{\Ltqpo(\Xi)}.
\end{align*}
Thus $(H_n^{\ed})$ and $(H_n^{\cd})$ converge in $\Ltqeps(\Xi)$
and altogether $(H_n)$ converges in $\Ltqeps(\Xi)$ as well.
\end{proof}

\begin{rem}
\label{HSGrem}
The use of Helmholtz decompositions and regular potentials in the proof of Lemma \ref{HSG}
demonstrates the main idea behind an elegant proof of a compact embedding.
This general idea carries over to proofs of compact embeddings related 
to other kinds of Hilbert complexes as well,
arising, e.g., in elasticity, general relativity, or biharmonic problems,
see for example \cite{paulyzulehnerbiharmonic}.
\end{rem}

\subsection{The Compact Embedding for Weak Lipschitz Domains}
\label{mcpweaklip}

The aim of this section is to transfer Lemma \ref{HSG} to arbitrary weak Lipschitz 
pairs $(\om,\Gamma_{\tau})$. 
To this end we will employ a technical lemma, whose proof 
is sketched in \cite[Section 3]{picardcomimb} and \cite[Remark 2]{wecktrace}.
We give a detailed proof in the appendix.
Let us consider the following situation:
Let $\Theta$, $\widetilde{\Theta}$ be two bounded domains in $\rN$
with boundaries $\Upsilon:=\p\Theta$, $\widetilde\Upsilon:=\p\widetilde\Theta$
and let $\Upsilon_0\subset\Upsilon$ be relatively open.
Moreover, let 
$$\phi:\Theta\to\widetilde\Theta,\qquad
\psi:=\phi^{-1}:\widetilde\Theta\to\Theta$$
be Lipschitz diffeomorphisms, this is,
$\phi\in\Czo(\Theta,\widetilde\Theta)$ and $\psi=\phi^{-1}\in\Czo(\widetilde\Theta,\Theta)$. 
Then $\widetilde\Theta=\phi(\Theta)$, $\widetilde\Upsilon=\phi(\Upsilon)$
and we define $\widetilde\Upsilon_0:=\phi(\Upsilon_0)$.

\begin{lem}[pull-back lemma for Lipschitz transformations]
\label{lemtrafo}
Let $E\in\cDgenc{q}{\Upsilon_0}(\Theta)$ resp. $\Dgenc{q}{\Upsilon_0}(\Theta)$ 
and $H\in\eps^{-1}\cDegenc{q}{\Upsilon_0}(\Theta)$ resp. $\eps^{-1}\Degenc{q}{\Upsilon_0}(\Theta)$
for an admissible transformation $\eps$ on $\Ltq(\Theta)$. Then
\begin{align*}
\psi^* E&\in\cDgenc{q}{\widetilde\Upsilon_0}(\widetilde\Theta)
\text{ resp. }\Dgenc{q}{\widetilde\Upsilon_0}(\widetilde\Theta)&
&\text{and}&
\ed\psi^* E&=\psi^*\ed E,\\
\psi^*H&\in\mu^{-1}\cDegenc{q}{\widetilde\Upsilon_0}(\widetilde\Theta)
\text{ resp. }\mu^{-1}\Degenc{q}{\widetilde\Upsilon_0}(\widetilde\Theta)&
&\text{and}&
\cd\mu\psi^* H&=\pm\star\ed\psi^*\star\eps H = \pm\star\psi^*\star\cd\eps H,
\end{align*}
where $\mu:=(-1)^{qN-1}\star\psi^*\star\eps\phi^*$ is an admissible transformation. Moreover, there exists $c>0$, independent of $E$ and $H$, such that
$$\norm{\psi^*E}_{\Dq(\widetilde\Theta)}\leq c\norm{E}_{\Dq(\Theta)}, \qquad 
\norm{\psi^*H}_{\mu^{-1}\Deq(\widetilde\Theta)}\leq c\norm{H}_{\eps^{-1}\Deq(\Theta)}.$$
\end{lem}

Let $(\om,\Gamma_{\tau})$ be a bounded weak Lipschitz pair 
as introduced in Definitions \ref{defilipmani} and \ref{defilipsubmani}.
We adjust Lemma \ref{lemtrafo} to our situation:
Let $U_1,\dots,U_K$ be an open covering of $\Gamma$
according to Definitions \ref{defilipmani} and \ref{defilipsubmani}
and set $U_{0}:=\om$. Therefore $U_0,\dots,U_K$ is an open covering of $\ol\om$.
Moreover let $\chi_k\in\Cic(U_k)$, $k\in\{0,\dots,K\}$, 
be a partition of unity subordinate to the open covering $U_0,\dots,U_K$.
Now suppose $k\in\{1,\dots,K\}$. We define
\begin{align*}
\om_k&:=U_k\cap\om,&
\Gamma_k&:=U_k\cap\Gamma,&
\Gamma_{\tau,k}&:=U_k\cap\Gamma_{\tau},&
\Gamma_{\nu,k}&:=U_k\cap\Gamma_{\nu},\\
\widehat\Gamma_{k}&:=\p\om_{k},&
\Sigma_k&:=\widehat\Gamma_{k}\setminus\Gamma,&
\widehat\Gamma_{\tau,k}&:=\textrm{int}(\Gamma_{\tau,k}\cup\ol\Sigma_k),&
\widehat\Gamma_{\nu,k}&:=\textrm{int}(\Gamma_{\nu,k}\cup\ol\Sigma_k),\\
&&
\sigma&:=\gamma\setminus\ol B_{0},&
\widehat\gamma_{\tau}&:=\textrm{int}(\gamma_{\tau}\cup\ol\sigma),&
\widehat\gamma_{\nu}&:=\textrm{int}(\gamma_{\nu}\cup\ol\sigma).
\end{align*}
Lemma \ref{lemtrafo} will from now on be used with 
$$\Theta:=\om_k,\quad
\widetilde\Theta:=\Xi,\qquad
\phi:=\phi_k:\om_{k}\to\Xi,\quad
\psi:=\psi_k:\Xi\to\om_{k}$$
and with one of the following cases
$$\Upsilon_0:=\Gamma_{\tau,k},\quad
\Upsilon_0:=\widehat\Gamma_{\tau,k},\quad
\Upsilon_0:=\Gamma_{\nu,k},\quad
\Upsilon_0:=\widehat\Gamma_{\nu,k}.$$
Then $\Upsilon=\widehat\Gamma_{k}$ and $\widetilde\Upsilon=\phi_{k}(\widehat\Gamma_{k})=\gamma$ 
as well as (depending on the respective case)
\begin{align*}
\widetilde\Upsilon_0
&=\phi_{k}(\Gamma_{\tau,k})
=\gamma_{\tau},&
\widetilde\Upsilon_0
&=\phi_{k}(\widehat\Gamma_{\tau,k})
=\widehat\gamma_{\tau},&
\gamma_{\tau}
&\in\{\emptyset,B_{0},B_{0,-}\},&
\gamma_{\nu}
&=\gamma\setminus\ol\gamma_{\tau},\\
\widetilde\Upsilon_0
&=\phi_{k}(\Gamma_{\nu,k})
=\gamma_{\nu},&
\widetilde\Upsilon_0
&=\phi_{k}(\widehat\Gamma_{\nu,k})
=\widehat\gamma_{\nu},&
\gamma_{\nu}
&\in\{\emptyset,B_{0},B_{0,+}\},&
\gamma_{\tau}
&=\gamma\setminus\ol\gamma_{\nu}.
\end{align*}

\begin{rem}
\label{Bminus}
Lemmas \ref{satzD0L6}, \ref{satzLtL6}, \ref{rss}, \ref{lemHDXi}, \ref{HSG}
hold for $\gamma_{\nu}=B_{0,-}$ without any (substantial) modification as well.
\end{rem}

It is straightforward to show the following:

\begin{lem}[localization]
\label{kor320}
Let $(\om,\Gamma_{\tau})$ be a bounded weak Lipschitz pair.
Then for $E\in\cDqct(\om)$, respectively, $E\in\Dqct(\om)$ and 
$H\in\cDeqcn(\om)$, respectively, $H\in\Deqcn(\om)$ we have
for $k\in\{1,\dots,K\}$
\begin{align*}
E&\in\cDgenc{q}{\Gamma_{\tau,k}}(\om_k),
&
\chi_k E&\in\cDgenc{q}{\widehat\Gamma_{\tau,k}}(\om_k),
&
H&\in\cDegenc{q}{\Gamma_{\nu,k}}(\om_k),
&
\chi_k H&\in\cDegenc{q}{\widehat\Gamma_{\nu.k}}(\om_k),\\
E&\in\Dgenc{q}{\Gamma_{\tau,k}}(\om_k),
&
\chi_k E&\in\Dgenc{q}{\widehat\Gamma_{\tau,k}}(\om_k),
&
H&\in\Degenc{q}{\Gamma_{\nu,k}}(\om_k),
&
\chi_k H&\in\Degenc{q}{\widehat\Gamma_{\nu.k}}(\om_k).
\end{align*}
\end{lem}

\begin{theo}[weak and strong boundary conditions coincide]
\label{theoweakeqstrong}
Let $(\om,\Gamma_{\tau})$ be a bounded weak Lipschitz pair.
Then $\cDqct(\om)=\Dqct(\om)$ and $\cDeqcn(\om)=\Deqcn(\om)$.
\end{theo}

\begin{proof}
Suppose $E\in\cDqct(\om)$. Then $\chi_{0}E\in\Dqc(\om)\subset\Dqct(\om)$ by mollification.
Let $k\in\{1,\dots,K\}$. Then $\chi_{k}E\in\cDgenc{q}{\widehat\Gamma_{\tau,k}}(\om_k)$ by Lemma \ref{kor320}. 
Lemma \ref{lemtrafo}, Lemma \ref{rss} (with $\gamma_{\nu}:=\gamma_{\tau}$) 
and Remark \ref{Bminus} yield
$$\psi_{k}^*(\chi_{k}E)\in\cDgenc{q}{\widehat\gamma_{\tau}}(\Xi)=\Dgenc{q}{\widehat\gamma_{\tau}}(\Xi),\qquad
\widehat\gamma_{\tau}=\phi_{k}(\widehat\Gamma_{\tau,k}),\qquad
\gamma_{\tau}\in\{\emptyset,B_{0},B_{0,-}\}.$$
Then $\chi_{k}E=\chi_{k}\phi_{k}^*\psi_{k}^*E\in\Dgenc{q}{\widehat\Gamma_{\tau,k}}(\om_k)\subset\Dqct(\om)$ 
by Lemma \ref{lemtrafo}.
Hence we see $E=\sum_{k}\chi_k E\in\Dqct(\om)$.
$\cDeqcn(\om)=\Deqcn(\om)$ follows analogously or by Hodge-$\star$-duality.
\end{proof}

Now the compact embedding for bounded weak Lipschitz pairs $(\om,\Gamma_{\tau})$ can be proved.

\begin{theo}[Weck's selection theorem]
\label{satzMKE}
Let $(\om,\Gamma_{\tau})$ be a bounded weak Lipschitz pair
and let $\eps$ be an admissible transformation on $\Ltqom$.
Then the embedding
$$\Dqct(\om)\cap\eps^{-1}\Deqcn(\om)\hookrightarrow\Ltqepsom$$
is compact.
\end{theo}

\begin{proof}
Suppose $(E_n)$ is a bounded sequence in $\Dqct(\om)\cap\eps^{-1}\Deqcn(\om)$. Then by mollification
$$E_{0,n}:=\chi_0E_n\in\Dqcom\cap\eps^{-1}\Deqcom$$
and $E_{0,n}$ even has compact support in $\om$.
By classical results, see \cite{weckmaxcomp,weckmax,picardcomimb}, 
$(E_{0,n})$ contains a subsequence,
again denoted by $(E_{0,n})$,
converging in $\Ltqepsom$.
Let $k\in\{1,\dots,K\}$. By Lemma \ref{kor320}
$$E_{k,n}:=\chi_kE_n\in\Dgenc{q}{\widehat\Gamma_{\tau,k}}(\om_k),\qquad
\eps E_{k,n}\in\Degenc{q}{\widehat\Gamma_{\nu,k}}(\om_k),$$
and the sequence $(E_{k,n})$ is bounded in 
$\Dgenc{q}{\widehat\Gamma_{\tau,k}}(\om_k)\cap\eps^{-1}\Degenc{q}{\widehat\Gamma_{\nu,k}}(\om_k)$ 
by the product rule. 
By Lemma \ref{lemtrafo} we have 
$\psi_{k}^* E_{k,n}\in\Dgenc{q}{\widehat\gamma_{\tau}}(\Xi)$ and
\begin{align*}
\norm{\psi_{k}^* E_{k,n}}_{\Dq(\Xi)} 
\leq c\norm{E_{k,n}}_{\Dq(\om_{k})},
\end{align*}
showing that $(\psi_{k}^* E_{k,n})$ is bounded in $\Dgenc{q}{\widehat\gamma_{\tau}}(\Xi)$. 
Analogously, $(\psi_{k}^*E_{k,n})\subset\mu_{k}^{-1}\Degenc{q}{\widehat\gamma_{\nu}}(\Xi)$ 
is bounded in $\mu_{k}^{-1}\Degenc{q}{\widehat\gamma_{\nu}}(\Xi)$ with the admissible transformation
$\mu_{k}:=(-1)^{qN-1}\star\psi_{k}^*\star\eps\phi_{k}^*$. 
Thus $(\psi_{k}^* E_{k,n})$ is bounded in 
$$\Dgenc{q}{\widehat\gamma_{\tau}}(\Xi)\cap\mu^{-1}_k\Degenc{q}{\widehat\gamma_{\nu}}(\Xi)
\subset\Dgenc{q}{\widehat\gamma_{\tau}}(\Xi)\cap\mu^{-1}_k\Degenc{q}{\gamma_{\nu}}(\Xi),\quad
\gamma_{\nu}\in\{\emptyset,B_{0},B_{0,+}\},\quad
\widehat\gamma_{\tau}=\gamma\setminus\ol\gamma_{\nu}.$$
Thus, by Lemma \ref{HSG} and without loss of generality, 
$(\psi_{k}^*E_{k,n})$ is a Cauchy sequence in $\Ltq(\Xi)$. Now
$$E_{k,n}=\phi_{k}^*\psi_{k}^*E_{k,n}\in\Ltq(\om_{k})$$
and Lemma \ref{lemtrafo} yields
\begin{align*}
\norm{E_{k,n}-E_{k,m}}_{\Ltq(\om_k)} 
\leq c\norm{\psi_{k}^*E_{k,n}-\psi_{k}^*E_{k,m}}_{\Ltq(\Xi)}.
\end{align*}
Hence $(E_{k,n})$ is a Cauchy sequence in $\Ltq(\om_k)$ and so in $\Ltqepsom$
for their extensions by zero to $\om$.
Finally, extracting convergent subsequences for $k=1,\dots,K$, we see that
$$(E_n)=\big(\sum_{k=0}^{K}\chi_kE_n\big)=\big(\sum_{k=0}^{K}E_{k,n}\big)$$
is a Cauchy sequence in $\Ltqepsom$.
\end{proof}

\section{Applications}
\label{sectApplications}

From now on, let $\om\subset\rN$ be a bounded domain and
let $(\om, \Gamma_{\tau})$ be a weak Lipschitz pair as well as $\eps:\Ltqom\rightarrow\Ltqom$ be admissible. 
Then by Theorem \ref{satzMKE} the embedding
\begin{align}
\label{compembappsec}
\Dqct(\om)\cap\eps^{-1}\Deqcn(\om)\hookrightarrow\Ltqom
\end{align}
is compact. The results of this section immediately follow in the framework 
of a general functional analytic toolbox, 
see \cite{paulyapostfirstordergen,paulydivcurl,paulyzulehnerbiharmonic}.
For details, see also the proofs in \cite{bauerpaulyschomburgmaxcompweaklip}
for the classical case of vector analysis.

\subsection{The Maxwell Estimate}
\label{secappmaxest}

A first consequence of \eqref{compembappsec}
is that the space of so-called ``harmonic'' Dirichlet-Neumann forms
$$\harmdiqeps:=\Dqczt(\om)\cap\eps^{-1}\Deqczn(\om)$$
is finite-dimensional, as the unit ball in $\harmdiqeps$ is compact by \eqref{compembappsec}.
By a standard indirect argument, 
\eqref{compembappsec} immediately implies the so-called Maxwell estimate:

\begin{theo}[Maxwell estimate]
\label{MA}
There is $c_{\mathsf{m}}>0$, 
such that for all $E\in\Dqct(\om)\cap\eps^{-1}\Deqcn(\om)\cap\harmdiqeps^{\perp_\eps}$
\begin{align*}
\norm{E}_{\Ltqepsom}
\leq c_{\mathsf{m}}\,\big(\normLtqpoom{\ed E}^2
+\normLtqmoom{\cd\eps E}^2\big)^{1/2}.
\end{align*}
\end{theo}

Here we denote by $\perp_{\eps}$ orthogonality with respect to the $\Ltqeps(\om)$-inner product.

\subsection{Helmholtz Decompositions}
\label{sechelmdeco}

Applying the projection theorem to the densely defined and closed (unbounded) linear operators
$$\ed_{\tau}^{q-1}:\Dqmoct(\om)\subset\Ltqmo(\om)\rightarrow\Ltqepsom\,;\quad
E\mapsto\ed E$$
with (Hilbert space) adjoint (see Theorem \ref{theoweakeqstrong})
$$-\cd^q_{\nu}\eps:=(\ed^{q-1}_{\tau})^*:\eps^{-1}\Deqcn(\om)\subset\Ltqepsom\rightarrow\Ltqmoom\,;\quad
H\mapsto-\cd\eps H$$
and 
$$-\eps^{-1}\cd^{q+1}_{\nu}:\eps^{-1}\Deqpocn(\om)\subset\Ltqpoom\rightarrow\Ltqepsom\,;\quad
H\mapsto-\eps^{-1}\cd H$$ 
with adjoint (see Theorem \ref{theoweakeqstrong})
$$\ed^{q}_{\tau}:=(-\eps^{-1}\cd^{q+1}_{\nu})^*:\Dqct(\om)\subset\Ltqepsom\rightarrow\Ltqpoom\,;\quad
E\mapsto\ed E$$
we obtain the Helmholtz decompositions
\begin{align}
\label{HDom1}
\Ltqeps(\om)
&=\overline{\ed\Dqmoct(\om)}\oplus_{\eps}\eps^{-1}\Deqczn(\om),\\
\label{HDom2}
\Ltqeps(\om)
&=\Dqczt(\om)\oplus_\eps\overline{\eps^{-1}\cd\Deqpocn(\om)}.
\end{align}
Therefore, $\Dqczt(\om)=\overline{\ed\Dqmoct(\om)} \oplus_{\eps}\harmdiqeps$
and, altogether, we get the refined Helmholtz decomposition
\begin{align}
\label{HDom3}
\Ltqeps(\om)=\overline{\ed\Dqmoct(\om)}\oplus_\eps\harmdiqeps\oplus_\eps\overline{\eps^{-1}\cd\Deqpocn(\om)}.
\end{align}

\begin{theo}[Helmholtz decompositions]
\label{HZ}
The orthonormal decompositions
\begin{align*}
\Ltqepsom
&=\ed\Dqmoct(\om)\oplus_\eps \eps^{-1}\Deqczn(\om)\\
&=\Dqczt(\om)\oplus_\eps\eps^{-1}\cd\Deqpocn(\om)\\
&=\ed\Dqmoct(\om)\oplus_\eps\harmdiqeps\oplus_\eps\eps^{-1}\cd\Deqpocn(\om)
\end{align*}
hold. Furthermore
\begin{align*}
\ed\Dqct(\om)
&=\ed\big(\Dqct(\om)\cap\eps^{-1}\cd\Deqpocn(\om)\big)
=\ed\big(\Dqct(\om)\cap\eps^{-1}\Deqczn(\om)\cap\harmdiqeps^{\perp_\eps}\big),\\
\cd\Deqcn(\om)
&=\cd\big(\Deqcn(\om)\cap\eps\ed\Dqmoct(\om)\big)
=\cd\Big(\Deqcn(\om)\cap\eps\big(\Dqczt(\om)\cap\harmdiqeps^{\perp_\eps}\big)\Big)
\end{align*}
and
\begin{align*}
\ed\Dqmoct(\om)
&=\Dqczt(\om)\cap\harmdiqeps^{\perp_\eps},&
\cd\Deqpocn(\om)
&=\Deqczn(\om)\cap\harmdiqeps^{\perp},\\
\Dqczt(\om)
&=\ed\Dqmoct(\om)\oplus_{\eps}\harmdiqeps,&
\Deqczn(\om)
&=\cd\Deqpocn(\om)\oplus_{\eps^{-1}}\eps\harmdiqeps.
\end{align*}
The ranges $\ed\Dqmoct(\om)$ and $\cd\Deqpocn(\om)$ are closed subspaces of $\Ltqepsom$.
Moreover,  the $\ed$- resp. $\cd$-potentials are uniquely determined
in $\Dqct(\om)\cap\eps^{-1}\Deqczn(\om)\cap\harmdiqeps^{\perp_\eps}$
and $\Deqcn(\om)\cap\eps\big(\Dqczt(\om)\cap\harmdiqeps^{\perp_\eps}\big)$, respectively,
and depend continuously on their respective images.
\end{theo}

\begin{proof}
For $\eps=\id$ \eqref{HDom1} and \eqref{HDom2} yield 
\begin{align*}
\Deqcn(\om)
&=\big(\overline{\ed\Dqmoct(\om)}\cap\Deqcn(\om)\big)
\oplus\Deqczn(\om),\\
\Dqct(\om)
&=\Dqczt(\om)
\oplus\big(\Dqct(\om)\cap\overline{\cd\Deqpocn(\om)}\big)
\end{align*}
and thus with \eqref{HDom1}, \eqref{HDom2}, and \eqref{HDom3}
\begin{align*}
\cd\Deqcn(\om)
&=\cd\big(\Deqcn(\om)\cap\ol{\ed\Dqmoct(\om)}\big)
=\cd\big(\Dqczt(\om)\cap\Deqcn(\om)\cap\harmdiq^{\perp}\big),\\
\ed\Dqct(\om)
&=\ed\big(\Dqct(\om)\cap\ol{\cd\Deqpocn(\om)}\big)
=\ed\big(\Dqct(\om)\cap\Deqczn(\om)\cap\harmdiq^{\perp}\big).
\end{align*}
Now Theorem \ref{MA} implies the closedness of the ranges 
and the continuity of the potentials. The other assertions follow immediately.
\end{proof}

\begin{cor}[refined Helmholtz decompositions]
\label{HZcor}
It holds
\begin{align*}
\Dqct(\om)
&=\ed\Dqmoct(\om)\oplus_\eps\big(\Dqct(\om)\cap\eps^{-1}\Deqczn(\om)\big)\\
&=\Dqczt(\om)\oplus_\eps\big(\Dqct(\om)\cap\eps^{-1}\cd\Deqpocn(\om)\big)\\
&=\ed\Dqmoct(\om)\oplus_\eps\harmdiqeps\oplus_\eps\big(\Dqct(\om)\cap\eps^{-1}\cd\Deqpocn(\om)\big),\\
\eps^{-1}\Deqcn(\om)
&=\big(\ed\Dqmoct(\om)\cap\eps^{-1}\Deqcn(\om)\big)\oplus_\eps\eps^{-1}\Deqczn(\om)\\
&=\big(\Dqczt(\om)\cap\eps^{-1}\Deqcn(\om)\big)\oplus_\eps\eps^{-1}\cd\Deqpocn(\om)\\
&=\big(\ed\Dqmoct(\om)\cap\eps^{-1}\Deqcn(\om)\big)\oplus_\eps\harmdiqeps\oplus_\eps\eps^{-1}\cd\Deqpocn(\om).
\end{align*}
\end{cor}

\subsection{Static Solution Theory}

As a further application we turn to the boundary value problem 
of generalized electro- and magnetostatics with mixed boundary values: 
Let $F\in\Ltqpoom$, $G\in\Ltqmoom$, 
$E_\tau,\,E_\nu\in\Ltqepsom$, and let $\eps$ be admissible. 
The problem is to find $E\in\Dq(\om)\cap\eps^{-1}\Deq(\om)$ such that
\begin{align}
\begin{split}
\label{EMS}
\ed E &= F,\\
\cd \eps E &= G,\\
E - E_{\tau} &\in\Dqct(\om),\\
\eps(E - E_{\nu}) &\in\Deqcn(\om).
\end{split}
\end{align}
For uniqueness, we require the additional conditions
\begin{align}
\label{ONBproj}
\scpLtqepsom{\eps E}{D_\ell} = \alpha_\ell \in\rz,\quad \ell=1,\dots,d,
\end{align}
where $d$ is the dimension and $\{D_\ell\}$ an $\eps$-orthonormal basis of $\harmdiqeps$.
The boundary values on $\Gamma_{\tau}$ and $\Gamma_{\nu},$ respectively, 
are realised by the given volume forms $E_{\tau}$ and $E_{\nu}$, respectively.

\begin{theo}[static solution theory]
\label{satzloesung}
\eqref{EMS} admits a solution, if and only if 
$$E_{\tau}\in\Dq(\om),\quad 
E_{\nu}\in\eps^{-1}\Deq(\om),$$
and
\begin{align}
\label{perploes}
F-\ed E_{\tau}\perp\Deqpoczn(\om),
\quad G-\cd \eps E_{\nu}\perp\Dqmoczt(\om).
\end{align}
The solution $E\in\Dq(\om)\cap\eps^{-1}\Deq(\om)$ can be chosen in a way
such that condition \eqref{ONBproj} with $\alpha\in\rz^d$ is fulfilled, 
which then uniquely determines the solution. 
Furthermore, the solution depends linearly and continuously on the data.
\end{theo}

Note that \eqref{perploes} is equivalent to 
$$F-\ed E_{\tau}\in\ed\Dqct(\om),\qquad
G-\cd\eps E_{\nu}\in\cd\Deqcn(\om).$$
For homogeneous boundary data, i.e., $E_{\tau}=E_{\nu}=0$, 
the latter theorem immediately follows from a functional analytic toolbox, 
see \cite{paulyapostfirstordergen,paulydivcurl,paulyzulehnerbiharmonic},
which even states a sharper result: The linear static Maxwell-operator
$$\Abb{M}{\Dqct(\om)\cap\eps^{-1}\Deqcn(\om)}{\ed\Dqct(\om)\times\cd\Deqcn(\om)\times\rz^d}
{E}{\big(\ed E,\cd\eps E,(\scpLtqepsom{\eps E}{D_\ell})_{\ell=1}^{d}\big)}$$
is a topological isomorphism. Its inverse $M^{-1}$ maps not only continuously
onto $\Dqct(\om)\cap\eps^{-1}\Deqcn(\om)$, but also compactly into $\Ltqepsom$ by \eqref{compembappsec}.
For homogeneous kernel data, i.e., for 
$$\Abb{M_{0}}{\Dqct(\om)\cap\eps^{-1}\Deqcn(\om)\cap\harmdiqeps^{\perp_\eps}}{\ed\Dqct(\om)\times\cd\Deqcn(\om)}
{E}{(\ed E,\cd\eps E)},$$
we have $\norms{M_{0}^{-1}}\leq(c_{\mathsf{m}}^2+1)^{1/2}$.
For details and a proof of Theorem \ref{satzloesung}
in the classical setting of vector analysis see \cite{bauerpaulyschomburgmaxcompweaklip}.

\section{Remarks on the Transformations}

Let us mention some observations on the transformations,
in particular that some results are independent 
of the admissible transformation $\eps$.
For this, let $\om\subset\rN$ be an open set,
let $\eps$ and $\mu$ be admissible transformations on $\Ltqom$, 
and let us recall the arguments leading to \eqref{HDom1}
in Section \ref{sechelmdeco} for the
densely defined and closed (unbounded) linear operator
$$\ed_{\tau}^{q-1}:\Dqmoct(\om)\subset\Ltqmo(\om)\rightarrow\Lgen{2,q}{\mu}(\om)\,;\quad
E\mapsto\ed E.$$
If Theorem \ref{theoweakeqstrong} (weak and strong boundary conditions coincide)
is not available, its adjoint is given by
$$-\cd^q_{\nu}\mu:=(\ed^{q-1}_{\tau})^*:\mu^{-1}\cDeqcn(\om)\subset\Lgen{2,q}{\mu}(\om)\rightarrow\Ltqmoom\,;\quad
H\mapsto-\cd\mu H$$
yielding instead of \eqref{HDom1} the Helmholtz (and refined Helmholtz) decompositions
\begin{align}
\label{HDom1weakL}
\Lgen{2,q}{\mu}(\om)
&=\overline{\ed\Dqmoct(\om)}\oplus_{\mu}\mu^{-1}\cDeqczn(\om),\\
\label{HDom1weakR}
\Dqct(\om)
&=\ol{\ed\Dqmoct(\om)}\oplus_{\mu}\big(\Dqct(\om)\cap\mu^{-1}\cDeqczn(\om)\big),\\
\label{HDom1weakD}
\mu^{-1}\cDeqcn(\om)
&=\big(\ol{\ed\Dqmoct(\om)}\cap\mu^{-1}\cDeqcn(\om)\big)\oplus_{\mu}\mu^{-1}\cDeqczn(\om).
\end{align}

\begin{theo}[independence of the transformation]
\label{epstheo}
Let $\om\subset\rN$ be an open set and let $\eps$ be an admissible transformation on $\Ltqom$.
\begin{itemize}
\item[\bf(i)]
Weck's selection theorem is independent of the transformation $\eps$, i.e.,
the compactness of the embedding in Theorem \ref{satzMKE} does not depend on $\eps$.
\item[\bf(ii)]
The dimension of $\harmdiqeps$ does not depend on $\eps$, 
in particular $\dim\harmdiqeps=\harmdiq$.
\item[\bf(iii)]
If Weck's selection theorem (Theorem \ref{satzMKE}) holds, 
then the dimension of $\harmdiqeps$ is finite.
\end{itemize}
\end{theo}

\begin{proof}
(iii) has already been shown in the beginning of Section \ref{secappmaxest}.

To show (i), let us assume that the embedding
\begin{align}
\label{comembproofepsind}
\Dqct(\om)\cap\mu^{-1}\cDeqcn(\om)\hookrightarrow\Lgen{2,q}{\mu}(\om)
\end{align}
is compact.
Moreover, let $(E_{n})$ be a bounded sequence in $\Dqct(\om)\cap\eps^{-1}\cDeqcn(\om)$.
By \eqref{HDom1weakR} we have
\begin{align}
\label{helmdecoproofepsmu}
\Dqct(\om)\ni E_{n}=E_{\ed,n}+E_{0,n}
\in\ol{\ed\Dqmoct(\om)}\oplus_{\mu}\big(\Dqct(\om)\cap\mu^{-1}\cDeqczn(\om)\big)
\end{align}
with $\ed E_{n}=\ed E_{0,n}$ and 
$\norm{E_{\ed,n}}_{\Lgen{2,q}{\mu}(\om)}\,,\,\norm{E_{0,n}}_{\Lgen{2,q}{\mu}(\om)}
\leq\norm{E_{n}}_{\Lgen{2,q}{\mu}(\om)}$. 
Hence $(E_{0,n})$ is a bounded sequence
in $\Dqct(\om)\cap\mu^{-1}\cDeqczn(\om)$ and therefore contains by \eqref{comembproofepsind}
a $\Lgen{2,q}{\mu}(\om)$-converging subsequence, again denoted by $(E_{0,n})$.
By \eqref{HDom1weakD} we get
$$\mu^{-1}\cDeqcn(\om)\ni\mu^{-1}\eps E_{n}=H_{\ed,n}+H_{0,n}
\in\big(\ol{\ed\Dqmoct(\om)}\cap\mu^{-1}\cDeqcn(\om)\big)\oplus_{\mu}\mu^{-1}\cDeqczn(\om)$$
with $\cd\eps E_{n}=\cd\mu H_{\ed,n}$ and
$\norm{H_{\ed,n}}_{\Lgen{2,q}{\mu}(\om)}\,,\,\norm{H_{0,n}}_{\Lgen{2,q}{\mu}(\om)}
\leq\norm{\mu^{-1}\eps E_{n}}_{\Lgen{2,q}{\mu}(\om)}$. 
Therefore $(H_{\ed,n})$ is a bounded sequence
in $\Dqczt(\om)\cap\mu^{-1}\cDeqcn(\om)$ and hence contains by \eqref{comembproofepsind}
a $\Lgen{2,q}{\mu}(\om)$-converging subsequence, again denoted by $(H_{\ed,n})$.
Then by orthogonality, i.e., $\ol{\ed\Dqmoct(\om)}\,\bot\,\cDeqczn(\om)$,
\begin{align*}
&\qquad\scpLtqepsom{E_{n}-E_{m}}{E_{n}-E_{m}}\\
&=\scpLtqom{\eps(E_{n}-E_{m})}{E_{\ed,n}-E_{\ed,m}}
+\scpLtqom{\eps(E_{n}-E_{m})}{E_{0,n}-E_{0,m}}\\
&=\scpLtqom{\mu(H_{\ed,n}-H_{\ed,m})}{E_{\ed,n}-E_{\ed,m}}
+\scpLtqom{\mu\mu^{-1}\eps(E_{n}-E_{m})}{E_{0,n}-E_{0,m}}\\
&\leq c\big(\norm{H_{\ed,n}-H_{\ed,m}}_{\Lgen{2,q}{\mu}(\om)}
+\norm{E_{0,n}-E_{0,m}}_{\Lgen{2,q}{\mu}(\om)}\big),
\end{align*}
which shows that $(E_{n})$ is a Cauchy sequence in $\Ltqepsom$.

To show (ii), we obtain by \eqref{HDom1weakR}
\begin{align}
\label{helmdecoproofepsmuedzero}
\Dqczt(\om)
=\ol{\ed\Dqmoct(\om)}\oplus_{\mu}\qharmdi{q}{\mu},\qquad
\qharmdi{q}{\mu}
=\Dqczt(\om)\cap\mu^{-1}\cDeqczn(\om),
\end{align}
and denote the orthonormal projector on the second component by $\pi$. Then
$$\Abb{\hat{\pi}}{\harmdiqeps}{\qharmdi{q}{\mu}}{H}{\pi H}$$
is injective, as $\hat{\pi}E=0$ implies $E\in\ol{\ed\Dqmoct(\om)}\cap\harmdiqeps=\{0\}$,
and hence $\dim\harmdiqeps\leq\dim\qharmdi{q}{\mu}$.
By symmetry we obtain $\dim\harmdiqeps=\dim\qharmdi{q}{\mu}$,
completing the proof.
\end{proof}

\section{General Regular Potentials and Decompositions}

A closer inspection of the proof of Lemma \ref{satzD0L6} shows that 
Lemma \ref{satzD0L6} and Lemma \ref{satzLtL6} hold for more general situations. 

\begin{defi}[extendable domain]
\label{defirotpot}
Let $\om\subset\rN$ and let $(\om,\Gamma_{\nu})$ be a bounded strong Lipschitz pair.
Moreover, let $\om$ and $\Gamma_{\nu}$ be topologically trivial (then so is $\Gamma_{\tau}$).
The pair $(\om,\Gamma_{\nu})$ is called ``extendable'', if
$\om$ can be extended through $\Gamma_{\nu}$ by zero to $\widehat{\om}$,
resulting in a topologically trivial strong Lipschitz domain 
$\widetilde{\om}=\text{\rm int}(\ol\om\cup\ol{\widehat\om})$.
\end{defi}

\begin{lem}[regular potentials and decompositions for extendable domains]
\label{regpotexted}
Let $\om\subset\rN$ and let $(\om,\Gamma_{\nu})$ be a 
bounded, topologically trivial, and extendable strong Lipschitz pair.
\begin{itemize}
\item[\bf(i)]
There exists a continuous linear operator
$$\S_{\ed}^{q}:\cDqczn(\om)\rightarrow\Hoqmo(\rN)\cap\Hoqmocn(\om),$$
such that $\ed\S_{\ed}^{q}=\id|_{\cDqczn(\om)}$, i.e.,
for all $H\in\cDqczn(\om)$
$$\ed\S_{\ed}^{q}H=H\quad\text{in }\om.$$
Especially 
$$\Dgenc{q}{\Gamma_{\nu},0}(\om)
=\cDqczn(\om)
=\ed\S_{\ed}^{q}\cDqmoczn(\om)
=\ed\Hoqmocn(\om)
=\ed\Dqmocn(\om)
=\ed\cDqmocn(\om)$$
and the regular $\Hoqmocn(\om)$-potential depends continuously on the data.
In particular, these spaces are closed subspaces of 
$\Ltq(\om)$ and $\S_{\ed}^{q}$ is a right inverse to $\ed$.
Without loss of generality, 
$\S_{\ed}$ maps to forms with a fixed compact support in $\rN$.
\item[\bf(ii)]
The regular decompositions
\begin{align*}
\cDgenc{q}{\Gamma_{\nu}}(\om)
=\Dqcn(\om)
&=\Hoqcn(\om)+\ed\Hoqmocn(\om)\\
&=\S_{\ed}^{q+1}\ed\Dqcn(\om)
\dotplus\ed\S_{\ed}^{q}(1-\S_{\ed}^{q+1}\ed)\Dqcn(\om)\\
&=\S_{\ed}^{q+1}\ed\Dqcn(\om)
\dotplus\Dgenc{q}{\Gamma_{\nu},0}(\om)
\end{align*}
hold with linear and continuous regular decomposition resp. potential operators,
where $\dotplus$ denotes the direct sum. More precisely,
$\S_{\ed}^{q+1}\ed+\ed\S_{\ed}^{q}(1-\S_{\ed}^{q+1}\ed)
=\id|_{\Dqcn(\om)}$, i.e., for all $E\in\Dqcn(\om)$
$$E=\S_{\ed}^{q+1}\ed E+\ed\S_{\ed}^{q}(1-\S_{\ed}^{q+1}\ed)E
\in\Hoqcn(\om)+\ed\Hoqmocn(\om)$$
with the linear and continuous regular potential operators
\begin{align*}
\S_{\ed}^{q+1}\ed:\Dqcn(\om)
&\to\Hoqcn(\om),\\
\S_{\ed}^{q}(1-\S_{\ed}^{q+1}\ed):\Dqcn(\om)
&\to\Hoqmocn(\om).
\end{align*}
\item[\bf(iii)]
Hodge-$\star$-duality yields the corresponding results for the co-derivative.
In particular, there exists a continuous linear $\cd$-right inverse operator
$$\S_{\cd}^{q}:\cDeqczn(\om)\rightarrow\Hoqpo(\rN)\cap\Hgenc{1,q+1}{\Gamma_{\nu}}(\om),$$
i.e., $\cd\S_{\cd}^{q}=id|_{\cDeqczn(\om)}$.
Moreover, 
$\Degenc{q}{\Gamma_{\nu},0}(\om)
=\cDeqczn(\om)
=\cd\Hgenc{1,q+1}{\Gamma_{\nu}}(\om)$
and the regular $\Hgenc{1,q+1}{\Gamma_{\nu}}(\om)$-potential depends continuously on the data.
Furthermore, the regular decompositions
\begin{align*}
\cDegenc{q}{\Gamma_{\nu}}(\om)
=\Degenc{q}{\Gamma_{\nu}}(\om)
&=\Hoqcn(\om)+\cd\Hgenc{1,q+1}{\Gamma_{\nu}}(\om)\\
&=\S_{\cd}^{q-1}\cd\Degenc{q}{\Gamma_{\nu}}(\om)
\dotplus\cd\S_{\cd}^{q}(1-\S_{\cd}^{q-1}\cd)\Degenc{q}{\Gamma_{\nu}}(\om)
\end{align*}
hold with the linear and continuous regular potential operators
\begin{align*}
\S_{\cd}^{q-1}\cd:\Degenc{q}{\Gamma_{\nu}}(\om)
&\to\Hoqcn(\om),\\
\S_{\cd}^{q}(1-\S_{\cd}^{q-1}\cd):\Degenc{q}{\Gamma_{\nu}}(\om)
&\to\Hgenc{1,q+1}{\Gamma_{\nu}}(\om),
\end{align*}
and 
$\S_{\cd}^{q-1}\cd+\cd\S_{\cd}^{q}(1-\S_{\cd}^{q-1}\cd)
=\id|_{\Degenc{q}{\Gamma_{\nu}}(\om)}$.
\end{itemize}
\end{lem}

\begin{proof}
For (i) we follow the proof of Lemma \ref{satzD0L6}.
To show (ii), we first note 
$\cDgenc{q}{\Gamma_{\nu}}(\om)=\Dqcn(\om)$ by Theorem \ref{theoweakeqstrong}.
Let $E\in\Dqcn(\om)$. Then $\ed E\in\Dgenc{q+1}{\Gamma_{\nu},0}(\om)$
and by (i) we see $\S_{\ed}^{q+1}\ed E\in\Hoqcn(\om)$ with $\ed(\S_{\ed}^{q+1}\ed E)=\ed E$.
Thus $E-\S_{\ed}^{q+1}\ed E\in\Dgenc{q}{\Gamma_{\nu},0}(\om)=\ed\Hoqmocn(\om)$ 
and $\S_{\ed}^{q}(E-\S_{\ed}^{q+1}\ed E)\in\Hoqmocn(\om)$
with $\ed\S_{\ed}^{q}(E-\S_{\ed}^{q+1}\ed E)=E-\S_{\ed}^{q+1}\ed E$
by (i), yielding
$$E=\S_{\ed}^{q+1}\ed E+\ed\S_{\ed}^{q}(1-\S_{\ed}^{q+1}\ed)E
\in\Hoqcn(\om)+\ed\Hoqmocn(\om),$$
which proves the regular decompositions and 
also the assertions about the regular potential operators.
To show the directness of the sums, let
$H=\S_{\ed}^{q+1}\ed E\in\Dgenc{q}{\Gamma_{\nu},0}(\om)$
with some $E\in\Dqcn(\om)$.
Then $0=\ed H=\ed E$ as $\ed E\in\Dgenc{q+1}{\Gamma_{\nu},0}(\om)$
and thus $H=0$.
\end{proof}

\begin{rem}[trivial Dirichlet-Neumann forms for extendable domains]
\label{regpotextedrem}
Let $\om\subset\rN$ and let $(\om,\Gamma_{\tau})$ be a 
bounded, topologically trivial, and extendable strong Lipschitz pair.
Then the Dirichlet-Neumann forms are trivial, i.e., $\harmdiqeps=\{0\}$,
which follows immediately by Theorem \ref{HZ} 
and, interchanging $\Gamma_{\tau}$ and $\Gamma_{\nu}$, Lemma \ref{regpotexted} (i) as
$\ed\Dqmoct(\om)
=\Dqczt(\om)
=\ed\Dqmoct(\om)\oplus_\eps\harmdiqeps$.
\end{rem}

Now, assume $(\om,\Gamma_{\tau})$ to be a bounded strong Lipschitz pair
and let us recall the partition of unity from Section \ref{mcpweaklip}.
After some possible adjustments, $U_{k}$ and $\chi_{k}$
can be chosen such that $(\om_{k},\widehat\Gamma_{\nu,k})$ 
is a bounded, topologically trivial, and extendable strong Lipschitz pair
for all $k=0,\dots,K$. Maybe $U_{0}$ 
has to be replaced by more neighbourhoods $U_{-L},\dots,U_{0}$ 
to ensure that all pairs $(\om_{k},\widehat\Gamma_{\nu,k})$, $k=-L,\dots,K$,
are topologically trivial. Note that for all ``inner'' indices $k=-L,\dots,0$
we have $\om_{k}=U_{k}$ as well as 
$\widehat\Gamma_{\nu,k}=\widehat\Gamma_{k}=\p\om_{k}=\p U_{k}$.
Then for $E\in\Dqcn(\om)$ we have
$\chi_{k}E\in\Dgenc{q}{\widehat\Gamma_{\nu,k}}(\om_{k})$
for all $k$ by Lemma \ref{kor320}.
Lemma \ref{regpotexted} (ii) shows the decomposition
\begin{align*}
\chi_{k}E=E_{k}+\ed H_{k}
\in\Hgenc{1,q}{\widehat\Gamma_{\nu,k}}(\om_{k})
+\ed\Hgenc{1,q-1}{\widehat\Gamma_{\nu,k}}(\om_{k})
\end{align*}
with potentials depending continuously on $\chi_{k}E$.
Extending $E_{k}$ and $H_{k}$ by zero to $\om$
yields $\widetilde{E}_{k}\in\Hoqcn(\om)$ 
and $\widetilde{H}_{k}\in\Hoqmocn(\om)$ and
$$E=\sum_{k}\chi_{k}E=\sum_{k}\widetilde{E}_{k}+\ed\sum_{k}\widetilde{H}_{k}
\in\Hoqcn(\om)+\ed\Hoqmocn(\om).$$
As all operations have been linear and continuous 
we obtain the regular decomposition and potential representation
\begin{align}
\label{regdecogenLip}
\Dqcn(\om)
=\Hoqcn(\om)+\ed\Hoqmocn(\om),\qquad
\ed\Dqcn(\om)
=\ed\Hoqcn(\om)
\end{align}
with linear and continuous potential operators
\begin{align*}
\P_{\ed}^{q}:\Dqcn(\om)&\to\Hoqcn(\om),
&
\S_{\ed}^{q+1}:\ed\Dqcn(\om)&\to\Hoqcn(\om),\\
\Q_{\ed}^{q}:\Dqcn(\om)&\to\Hoqmocn(\om).
\end{align*}
Note that by Theorem \ref{HZ}
\begin{align}
\label{hemldecospdeco}
\ed\Dqcn(\om)
=\Dgenc{q+1}{\Gamma_{\nu},0}{}(\om)\cap\qharmdi{q+1}{\eps}^{\perp_{\eps}},\qquad
\Dgenc{q+1}{\Gamma_{\nu},0}{}(\om)
=\ed\Dqcn(\om)\oplus_{\eps}\qharmdi{q+1}{\eps},
\end{align}
where here $\Gamma_{\tau}$ and $\Gamma_{\nu}$ are interchanged in the definition of
$$\qharmdi{q+1}{\eps}
:=\Dgenc{q+1}{\Gamma_{\nu},0}{}(\om)\cap\eps^{-1}\Degenc{q+1}{\Gamma_{\tau},0}{}(\om).$$

Let us summarise the results related to \eqref{regdecogenLip}.

\begin{theo}[regular potentials and decompositions for $\ed$ in strong Lipschitz domains]
\label{regpotextedtheo}
Let $\om\subset\rN$ and let $(\om,\Gamma_{\nu})$ be a 
bounded strong Lipschitz pair.
\begin{itemize}
\item[\bf(i)]
There exists a continuous linear operator
$$\S_{\ed}^{q}:\ed\Dgenc{q-1}{\Gamma_{\nu}}(\om)\rightarrow\Hoqmocn(\om),$$
such that $\ed\S_{\ed}^{q}=\id|_{\ed\Dgenc{q-1}{\Gamma_{\nu}}(\om)}$.
Especially 
$$\ed\Dqmocn(\om)
=\ed\Hoqmocn(\om)$$
and the regular $\Hoqmocn(\om)$-potential depends continuously on the data.
In particular, these spaces are closed subspaces of 
$\Ltq(\om)$ and $\S_{\ed}^{q}$ is a right inverse to $\ed$.
\item[\bf(ii)]
The regular decompositions
\begin{align*}
\Dqcn(\om)
&=\Hoqcn(\om)
+\ed\Hoqmocn(\om)
&
\Dgenc{q}{\Gamma_{\nu},0}(\om)
&=\ed\Hoqmocn(\om)
+\big(\Hoqcn(\om)\cap\Dgenc{q}{\Gamma_{\nu},0}(\om)\big)\\
&=\S_{\ed}^{q+1}\ed\Dqcn(\om)
\dotplus\Dgenc{q}{\Gamma_{\nu},0}(\om),
&
&=\ed\Hoqmocn(\om)
\oplus\harmdiq\\
&&
&=\ed\Hoqmocn(\om)
\oplus_{\eps}\harmdiqeps
\end{align*}
hold with linear and continuous regular decomposition resp. potential operators,
which can be defined explicitly by the orthonormal Helmholtz projectors
and the operators $\S_{\ed}^{q}$.
\end{itemize}
\end{theo}

\begin{proof}
(i) and the first regular decomposition of (ii)
together with the existence of the regular potential operators
are clear from the considerations leading to \eqref{regdecogenLip}.
Let $E\in\Dqcn(\om)$. As $\ed\S_{\ed}^{q+1}\ed E=\ed E$ by (i),
we have $E-\S_{\ed}^{q+1}\ed E\in\Dgenc{q}{\Gamma_{\nu},0}(\om)$,
showing the second regular decomposition of (ii).
As in the proof of Lemma \ref{regpotexted} the sum is direct.
Finally, (i) and \eqref{hemldecospdeco} complete the proof.
\end{proof}

\begin{rem}
Note that $\harmdiq$ is a subspace of smooth forms, i.e.,
$$\harmdiq
=\Dgenc{q}{\Gamma_{\nu},0}{}(\om)
\cap\Degenc{q}{\Gamma_{\tau},0}{}(\om)
\cap\Ciqom.$$
\end{rem}

Hodge-$\star$-duality yields the corresponding results for the co-derivative.

\begin{theo}[regular potentials and decompositions for $\cd$ in strong Lipschitz domains]
\label{regpotextcdtheo}
Let $\om\subset\rN$ and let $(\om,\Gamma_{\nu})$ be a 
bounded strong Lipschitz pair.
\begin{itemize}
\item[\bf(i)]
There exists a continuous linear operator
$$\S_{\cd}^{q}:\cd\Degenc{q+1}{\Gamma_{\nu}}(\om)\rightarrow\Hgenc{1,q+1}{\Gamma_{\nu}}{}(\om),$$
such that $\cd\S_{\cd}^{q}=\id|_{\cd\Degenc{q+1}{\Gamma_{\nu}}(\om)}$.
Especially 
$$\cd\Degenc{q+1}{\Gamma_{\nu}}(\om)
=\cd\Hgenc{1,q+1}{\Gamma_{\nu}}{}(\om)$$
and the regular $\Hgenc{1,q+1}{\Gamma_{\nu}}{}(\om)$-potential depends continuously on the data.
In particular, these spaces are closed subspaces of 
$\Ltq(\om)$ and $\S_{\cd}^{q}$ is a right inverse to $\cd$.
\item[\bf(ii)]
The regular decompositions
\begin{align*}
\Deqcn(\om)
&=\Hoqcn(\om)
+\cd\Hgenc{1,q+1}{\Gamma_{\nu}}{}(\om)
&
\Degenc{q}{\Gamma_{\nu},0}(\om)
&=\cd\Hgenc{1,q+1}{\Gamma_{\nu}}{}(\om)
+\big(\Hoqcn(\om)\cap\Degenc{q}{\Gamma_{\nu},0}(\om)\big)\\
&=\S_{\cd}^{q-1}\cd\Deqcn(\om)
\dotplus\Degenc{q}{\Gamma_{\nu},0}(\om),
&
&=\cd\Hgenc{1,q+1}{\Gamma_{\nu}}{}(\om)
\oplus\harmdiq\\
&&
&=\cd\Hgenc{1,q+1}{\Gamma_{\nu}}{}(\om)
\oplus\harmdiqeps
\end{align*}
hold with linear and continuous regular decomposition resp. potential operators,
which can be defined explicitly by the orthonormal Helmholtz projectors
and the operators $\S_{\cd}^{q}$.
\end{itemize}
\end{theo}

In the latter theorem for $\cd$
the Dirichlet-Neumann forms have again the usual boundary conditions
$$\harmdiqeps=\Dqczt(\om)\cap\eps^{-1}\Deqczn(\om).$$

For the case of no or full boundary conditions, related results
on regular potentials and regular decompositions are presented in 
\cite{costabelmcintoshbogopoinlipop}.


\bibliographystyle{plain} 
\bibliography{paule-michael}


\appendix
\section{Proof of Lemma \ref{lemtrafo} (Pull-Back Lemma for Lipschitz Transformations)}

We start out by proving the assertions for the exterior derivative.

\subsection{Without Boundary Conditions}
\label{appendixnoBC}

Let $E=\sum_{I}E_{I}\ed x^I \in\Dq(\Theta)$. 
We have to show $\psi^* E \in \Dq(\widetilde\Theta)$ with $\ed \psi^* E= \psi^* \ed E$.
\begin{itemize}
\item[\bf(i)]
Let us first consider $\Phi=\sum_{I}\Phi_{I}\ed x^I\in\Czoq(\Theta)$, i.e., 
$\Phi_{I}\in\Czo(\Theta)$ for all $I$. 
In the following we denote by $\widetilde\cdot$ the composition with $\psi$. We have
\begin{align*}
\ed\psi_{j}
&= \sum_{i}\partial_{i}\psi_{j}\ed x^i, 
&
\psi^* \Phi 
&= \sum_{I}\widetilde \Phi_{I}\psi^*\ed x^I 
= \sum_{I}\widetilde \Phi_{I}(\ed\psi_{i_{1}})\wedge\cdot\cdot\cdot\wedge(\ed\psi_{i_{q}}),\\
&&
\ed \Phi &=\sum_{I,j}\partial_{j}\Phi_{I}(\ed x_{j})\wedge(\ed x^I).
\end{align*}
By Rademacher's theorem $\widetilde \Phi_{I} = \Phi_{I}\circ\psi$ and $\psi_{j}$ 
belong to $\Czo(\widetilde\Theta)\subset\Ho(\widetilde\Theta)$
and the chain rule holds, i.e., 
$\partial_{i}\widetilde \Phi_{I}=\sum_{j}\widetilde{\partial_{j}\Phi_{I}}\partial_{i}\psi_{j}$.
As $\psi_{j}\in\Ho(\widetilde\Theta)$ 
we get $\ed\psi_{j}\in\DSobolev^1_{0}(\widetilde\Theta)$ by
$$\scp{\ed\psi_{j}}{\cd\varphi}_{\Lebesgue^{2,1}(\widetilde\Theta)}
=-\scp{\psi_{j}}{\cd\cd\varphi}_{\Lebesgue^{2,0}(\widetilde\Theta)}=0$$
for all $\varphi\in\Citc(\widetilde\Theta) $.
Thus by definition we see
\begin{align*}
	\ed\psi^* \Phi &=\sum_{I}(\ed \widetilde \Phi_{I})\wedge(\ed \psi_{i_{1}})\wedge\cdot\cdot\cdot\wedge (\ed\psi_{i_{q}}) 
	= \sum_{I,i}\partial_{i}\widetilde \Phi_{I}(\ed x^i)\wedge(\ed\psi_{i_{1}})\wedge\cdot\cdot\cdot\wedge (\ed\psi_{i_{q}})\\
	&=\sum_{I,i,j}\widetilde{\partial_{j} \Phi_{I}}\partial_{i}\psi_{j}(\ed x^i)\wedge(\ed\psi_{i_{1}})\wedge\cdot\cdot\cdot\wedge (\ed\psi_{i_{q}})
	=\sum_{I,j}\widetilde{\partial_{j} \Phi_{I}}(\ed\psi_{j})\wedge(\ed\psi_{i_{1}})\wedge\cdot\cdot\cdot\wedge (\ed\psi_{i_{q}}).
\end{align*}
On the other hand it holds
$$
\psi^*\ed \Phi = \sum_{I,j}\widetilde{\partial_{j} \Phi_{I}}(\psi^*\ed x_{j})\wedge(\psi^*\ed x^I)=\sum_{I,j}\widetilde{\partial_{j} \Phi_{I}}
(\ed\psi_{j})\wedge(\ed\psi_{i_{1}})\wedge\cdot\cdot\cdot\wedge (\ed\psi_{i_{q}}).
$$
Therefore, $\psi^*\Phi\in\Dq(\widetilde\Theta)$ and $\ed\psi^* \Phi = \psi^*\ed \Phi$. 
\item[\bf(ii)]
For general $E\in\Dq(\Theta)$ we pick $\Phi\in\Ciqpoc(\widetilde\Theta)$. 
Note $\supp\Phi\subset\subset\widetilde\Theta=\phi(\Theta)$. 
Replacing $\psi$ by $\phi$ in (i) we have $\phi^*\star\Phi\in\DNmqmo(\Theta)$ 
with $\ed\phi^*\star\Phi = \phi^*\ed\star\Phi$ and,
since $\phi^*\star\Phi = \sum_{I}\widetilde{(\star\Phi)_{I}}\phi^*\ed x^I$ holds, 
$\supp\phi^*\star\Phi\subset\subset\Theta$.
By standard mollification we obtain a sequence $(\Psi_{n})\subset\CiNmqmoc(\Theta)$ with $\Psi_{n}\rightarrow\phi^*\star\Phi$ in
$\DNmqmo(\Theta)$. Furthermore $\star\Psi_{n}\in\Ciqpoc(\Theta)$. Then
\begin{align*}
\scp{\psi^*E}{\cd\Phi}_{\Ltq(\widetilde\Theta)}
&=\int_{\widetilde\Theta} \psi^*E\wedge\star\cd\Phi 
	= \pm\int_{\widetilde\Theta}\psi^*E\wedge\psi^*\phi^*\ed\star\Phi 
	= \pm \int_{\widetilde\Theta}\psi^*(E\wedge\phi^*\ed\star\Phi)\\
	&= \pm\int_{\Theta} E\wedge\phi^* \ed\star\Phi 
	=\pm\int_{\Theta} E\wedge\ed\phi^*\star\Phi
	\leftarrow \pm\int_{\Theta} E\wedge\ed\Psi_{n}\\
	&=\pm\int_{\Theta} E\wedge\star\star\ed\star\star\Psi_{n}
	=\pm\scp{E}{\cd\star\Psi_{n}}_{\Ltq(\Theta)}\\
	&=\pm\scp{\ed E}{\star\Psi_{n}}_{\Ltqpo(\Theta)}
	\rightarrow\pm\scp{\ed E}{\star\phi^*\star\Phi}_{\Ltqpo(\Theta)}
	=\pm\int_{\Theta}\ed E\wedge\phi^*\star\Phi\\
	&=\pm\int_{\widetilde\Theta}\psi^*(\ed E\wedge\phi^*\star\Phi)
	=\pm\int_{\widetilde\Theta}(\psi^*\ed E)\wedge\star\Phi
	=-\scp{\psi^*\ed E}{\Phi}_{\Ltqpo(\widetilde\Theta)}
\end{align*}
and hence $\psi^* E\in\Dq(\widetilde\Theta)$ with $\ed\psi^*E = \psi^*\ed E$.
\item[\bf(iii)]
Let $E\in\Dq(\Theta)$. By (ii) we know $\psi^* E\in\Dq(\widetilde\Theta)$ with $\ed\psi^*E = \psi^*\ed E$. Hence
\begin{align*}
\norm{\psi^* E}^2_{\Ltq(\widetilde\Theta)}  
&=\int_{\widetilde\Theta}\psi^* E\wedge \star \psi^* E
=\int_{\Theta}\phi^*\psi^* E\wedge \phi^*\star \psi^* E\\
&=\pm\int_{\Theta} E\wedge \star(\star\phi^*\star\psi^*) E
\leq c\norm{E}^2_{\Ltq(\Theta)}
\end{align*}
and
\begin{align*}
\norm{\ed\psi^* E}_{\Ltqpo(\widetilde\Theta)} = \norm{\psi^*\ed E}_{\Ltqpo(\widetilde\Theta)}\leq c\norm{\ed E}_{\Ltqpo(\Theta)}.
\end{align*}
\end{itemize}
\subsection{With Strong Boundary Condition.}
\label{appendixsBC}
Let $E\in\Dgenc{q}{\Upsilon_0}(\Theta)$ and $(E_{n})\subset\Cgenc{\infty,q}_{\Upsilon_0}(\Theta)$ with $E_{n}\rightarrow E$ in $\Dq(\Theta)$.
By Appendix \ref{appendixnoBC} (ii) we know 
$\psi^*E_{n},\psi^*E\in\Dq(\widetilde\Theta)$ with $\ed\psi^*E_{n}=\psi^*\ed E_{n}$ 
as well as $\ed\psi^*E = \psi^*\ed E$.
Furthermore, $\psi^* E_{n}$ has compact support away from $\widetilde\Upsilon_{0}$.
By standard mollification we see $\psi^*E_{n}\in\Dgenc{q}{\widetilde\Upsilon_0}(\widetilde\Theta)$. Moreover, by \ref{appendixnoBC} (iii) $\psi^*E_{n} \rightarrow\psi^*E$ in $\Dq(\widetilde\Theta)$.
Therefore $\psi^* E\in\Dgenc{q}{\widetilde\Upsilon_0}(\widetilde\Theta)$ with $\ed\psi^* E = \psi^*\ed E$.

\subsection{With Weak Boundary Condition}
\label{appendixwBC}

Let $E\in\cDgenc{q}{\Upsilon_0}(\Theta)$ and $\Phi\in\Cgenc{\infty,q+1}_{\widetilde\Upsilon_1}(\widetilde\Theta)$,
where $\Upsilon_{1} = \Upsilon\setminus\ol\Upsilon_{0}$.
By Appendix \ref{appendixnoBC} (ii) we again know $\psi^*E\in\Dq(\widetilde\Theta)$ with $\ed\psi^*E=\psi^*\ed E$. Moreover by Appendix \ref{appendixsBC}
$\phi^*\star\Phi\in\Dgenc{N-q-1}{\Upsilon_{1}}(\Theta)$ and hence $\star\phi^*\star\Phi\in\Degenc{q+1}{\Upsilon_{1}}(\Theta)$. 
We repeat the calculation from Appendix \ref{appendixnoBC} (ii) to arrive at
\begin{align*}
\scp{\psi^*E}{\cd\Phi}_{\Ltq(\widetilde\Theta)} &=\int_{\widetilde\Theta} \psi^*E\wedge\star\cd\Phi = \pm\scp{E}{\star\phi^*\ed\star\Phi}_{\Ltq(\Theta)}\\
&=\pm\scp{E}{\star\ed\phi^*\star\Phi}_{\Ltq(\Theta)}=\pm\scp{E}{\cd\star\phi^*\star\Phi}_{\Ltq(\Theta)}\\
&=\pm\scp{\ed E}{\star\phi^*\star\Phi}_{\Ltqpo(\Theta)} = -\scp{\psi^*\ed E}{\Phi}_{\Ltqpo(\widetilde\Theta)} =-\scp{\ed\psi^*E}{\Phi}_{\Ltqpo(\widetilde\Theta)}
\end{align*}
and therefore $\psi^* E\in\cDgenc{q}{\widetilde\Upsilon_0}(\widetilde\Theta)$.

\subsection{Assertions for the Co-Derivative}
\label{appendixcoderi}

It holds by Appendix \ref{appendixnoBC} (ii)
\begin{align*}
\eps H\in\Deq(\Theta) \qequi \star\eps H\in\Dgen{N-q}{}{}(\Theta) \qequi \psi^*\star\eps \phi^*\psi^* H\in\Dgen{N-q}{}{}(\widetilde\Theta) \qequi \mu\psi^* H\in\Deq(\widetilde\Theta). 
\end{align*}
Moreover, using Appendix \ref{appendixnoBC} (iii) $\mu$ is admissible since for all $H\in\Ltq(\widetilde\Theta)$
\begin{align*}
\scp{\mu H}{H}_{\Ltq(\widetilde\Theta)} &= \pm\scp{\star\psi^*\star\eps\phi^* H}{H}_{\Ltq(\widetilde\Theta)}
=\pm\scp{\psi^*\star\eps\phi^* H}{\star H}_{\LtNmq(\widetilde\Theta)}\\
&=\pm\int_{\widetilde\Theta}\psi^*\star\eps\phi^* H \wedge H
=\pm\int_{\Theta} \star\eps\phi^* H \wedge \star\star\phi^* H\\
&=\pm\scp{\eps\phi^* H}{\phi^* H}_{\Ltq(\Theta)}
\geq c\norm{\phi^* H}^2_{\Ltq(\Theta)} \geq c\norm{H}^2_{\Ltq(\widetilde\Theta)}.
\end{align*}
Furthermore
\begin{align*}
\cd\mu\psi^* H&=\pm\star\ed\psi^*\star\eps H = \pm\star\psi^*\star\cd\eps H.
\end{align*}
The remaining assertions now follow by Appendix \ref{appendixnoBC}-\ref{appendixwBC} 
and Hodge-$\star$-duality.


\end{document}